\documentclass[10pt]{article}

\DeclareMathAlphabet{\mathpzc}{OT1}{pzc}{m}{it}

\usepackage{amsmath}
\usepackage{amssymb}
\usepackage{amsfonts}
\usepackage{amsthm}
\usepackage{mathrsfs}
\usepackage{ifsym}
\usepackage{fontenc}
\usepackage{textcomp}
\usepackage{tipa}
\usepackage{enumerate}
\usepackage[pdftex]{color, graphicx}
\numberwithin{equation}{section}
\usepackage{hyperref}

\begin{document}
 
\title{{\bf  Algebraic functions and class number formulas}}         
\author{Sushmanth J. Akkarapakam and Patrick Morton}        
\date{Nov. 1, 2025}          
\maketitle

\begin{abstract}  A class number formula is proved for extended ring class fields $L_{\mathcal{O},9}$ over imaginary quadratic fields $K_d = \mathbb{Q}(\sqrt{-d})$, in which the prime $p = 3$ splits, by determining the fields generated by the periodic points of a well-chosen algebraic function.  The number of periodic points of a given period $n \ge 2$ for this algebraic function equals six times the sum of class numbers of imaginary quadratic orders $\textsf{R}_{-d}$, for which the Artin symbol for a prime ideal divisor $\wp_3$ in $K_d$ of $3$ has order $n$ in the Galois group of $F_d/K_d$, where $F_d$ is the inertia field of $\wp_3$ in $L_{\mathcal{O},9}/K_d$.
\end{abstract}

\section{Introduction.}

In the well-known papers \cite{d1, d2}, Deuring used his theory of supersingular elliptic curves to derive an interesting class number formula, which connects the class numbers of different imaginary quadratic orders, and which can be stated as follows.  Let $\Omega_f$ be the ring class field of ring-conductor $f$ over the imaginary quadratic field $K_d = K = \mathbb{Q}(\sqrt{-d})$, where $-d = d_K f^2$ and $d_K$ is the discriminant of $K$.  For a given prime $p$ let $\mathfrak{D}_n^{(p)}$ denote the set of negative discriminants
$$\mathfrak{D}_n^{(p)} = \left\{-d = d_K f^2 \ | \ \left(\frac{-d}{p}\right) = +1 \ \wedge \ ord(\tau_\frak{p}) = n, \tau_\frak{p} = \left(\frac{\Omega_f/K}{\frak{p}}\right)\right\}.$$
Here, $\frak{p}$ is a first degree prime ideal divisor of $p$ in the ring of integers $R_K$ of $K$.  Then if all supersingular invariants in characteristic $p$ lie in the prime field $\mathbb{F}_p$, Deuring's formulas are equivalent to the formulas ($\mu(n)$ denotes the M\"obius $\mu$-function):
\begin{equation}
\sum_{-d \in \mathfrak{D}_{n}^{(p)}}{h(-d)} = \sum_{k \mid n}{\mu(n/k)p^k}, \ \ n>1.
\label{eqn:1.1}
\end{equation}

In two previous papers, an extension of Deuring's formulas was proved for the primes $p = 2, 5$.  To state these formulas, we recall the following from \cite{ch,co, so}.
If $\mathcal{O} = \textsf{R}_{-d}$ is the order of discriminant $-d$ in $R_K$, then $L_{\mathcal{O},m}$ denotes the {\it extended} ring class field over $K$\footnote{Cox's term in \cite{co}. This could also be called an {\it extended ray class field}, since it extends the ray class field $\Sigma_m$ by the ring class field $\Omega_{fm}$.}, whose corresponding ideal group in $K$ is
$$P_{\mathcal{O},m} = \{(\alpha) = \alpha R_K \ | \ \alpha \equiv a \ \textrm{mod} fm \ \textrm{in} \ R_K, \ a \in \mathbb{Z}, (a, f) = 1, a \equiv 1 \ \textrm{mod} \ m\},$$
in which the generators of principal ideals lie in the order $\mathcal{O}$.  See \cite[p. 313]{co}.  The field $L_{\mathcal{O},m}$ is a normal extension of the ring class field $\Omega_f = L_{\mathcal{O},1}$.  In general, the fact that $P_{\mathcal{O},m}$ is the intersection
$$P_{\mathcal{O},m} =  \{(\alpha) \ | \ \alpha \equiv a \ \textrm{mod} \ fm, a \in \mathbb{Z}, (a, fm) = 1\} \cap \{(\alpha) \ | \ \alpha \equiv 1 \ \textrm{mod} \ m\}$$
implies that
$$L_{\mathcal{O},m} = \Sigma_m \Omega_{fm},$$
where $\Sigma_m$ is the ray class field of conductor $(m)$.
(See \cite[Satz 11, p. 28]{h0}, \cite[p. 136]{chil}.)  If $K \neq \mathbb{Q}(\sqrt{-3})$ or $\mathbb{Q}(\sqrt{-4})$ and $(f,m) = 1$, then $\Omega_{fm} = \Omega_f \Omega_m$ (see \cite[Satz 3]{h}), so that
$$L_{\mathcal{O},m} = \Sigma_m \Omega_{fm} = \Sigma_m  \Omega_m \Omega_f = \Sigma_m \Omega_f.$$

Now we can state the extensions of Deuring's formulas. \begin{enumerate}[(i)]
\item ($p = 5$, \cite{m2}) If $\mathfrak{D}_{n,5}$ is the finite set of negative discriminants $-d \equiv 1, 4$ (mod $5$) for which $\tau_5 = \left(\frac{F_d/K_d}{\wp_5}\right)$ has order $n$ in $\textrm{Gal}(F_d/K_d)$, where $-d = d_K f^2$, $(5) = \wp_5 \wp_5'$ in $K_d$, and $F_d$ is the inertia field for the prime ideal $\wp_5$ in the extended ring class field $L_{\mathcal{O},5} = \Sigma_5 \Omega_{5f}$ over $K_d$, then
\begin{equation}
\sum_{-d \in \mathfrak{D}_{n,5}}{h(-d)} = \frac{1}{2} \sum_{k \mid n}{\mu(n/k)5^k}, \ \ n>1.
\label{eqn:1.2}
\end{equation}

\item ($p = 2$, \cite{am1}) If $\mathfrak{D}_{n,2}$ is the finite set of negative discriminants $-d \equiv 1$ (mod $8$) for which the Frobenius automorphism $\tau_2 = \left(\frac{\Sigma_{\wp_2'^3}\Omega_f/K_d}{\wp_2}\right)$  has order $n$ in $\textrm{Gal}(\Sigma_{\wp_2'^3}\Omega_f/K_d)$, where $(2) = \wp_2 \wp_2'$ in $K_d$ and $\Sigma_{\wp_2'^3}\Omega_f$ is the inertia field for $\wp_2$ in $L_{\mathcal{O},8} = \Sigma_8 \Omega_f$, then
\begin{equation}
\sum_{-d \in \mathfrak{D}_{n,2}}{h(-d)} = \frac{1}{2} \sum_{k \mid n}{\mu(n/k) 2^k}, \ \ n > 1.
\label{eqn:1.3}
\end{equation}
\end{enumerate}

In this paper we will prove a similar extension for the prime $p = 3$.  Let $K = K_d = \mathbb{Q}(\sqrt{-d})$ be an imaginary quadratic field whose discriminant $d_K$ is related to the integer $d$ by
$$-d = d_K f^2 \equiv 1 \ (\textrm{mod} \ 3).$$
Let $(3) = \wp_3 \wp_3'$ be the ideal factorization of $(3)$ in the ring of integers $R_K$ of $K$.  Further, let $\Sigma_{\wp_3'^2}$ be the ray class field of conductor
$$\mathfrak{f} = \wp_3'^2$$
over $K$, and let $\Omega_f$ denote the ring class field of conductor $f$ over $K$.  We will prove the following.  As above, $\mu(n)$ denotes the M\"obius $\mu$-function.

\newtheorem{thm}{Theorem}

\begin{thm}
If $\mathfrak{D}_{n,3}$ is the set of discriminants $-d \equiv 1$ (mod $3$) for which $\tau_3 = \left(\frac{\Sigma_{\wp_3'^2}\Omega_f/K_d}{\wp_3}\right)$ has order $n$ in $\textrm{Gal}(\Sigma_{\wp_3'^2}\Omega_f/K_d)$, where $-d = d_K f^2$ and $\Sigma_{\wp_3'^2}\Omega_f$ is the inertia field for $\wp_3$ in $L_{\mathcal{O},9} = \Sigma_9 \Omega_f$, then
\begin{equation}
\sum_{-d \in \mathfrak{D}_{n,3}}{h(-d)} = \frac{1}{3} \sum_{k \mid n}{\mu(n/k)3^k}, \ \ n>1.
\label{eqn:1.4}
\end{equation}
\label{thm:1}
\end{thm}

The discriminants $-d \in \mathfrak{D}_{n,3}$ are those for which $n$ is the smallest positive integer for which there is a primitive solution $(x,y) \in \mathbb{Z} \times \mathbb{Z}$ of the diophantine equation
$$4 \cdot 3^n = x^2+dy^2, \ \textrm{with} \ x \equiv \pm 1 \ (\textrm{mod} \ 9).$$ 
See Theorem \ref{thm:4} and the examples in Section 4.  \medskip

In \eqref{eqn:1.3} and \eqref{eqn:1.4} the conductors $\mathfrak{f} = \wp_3'^2, \wp_2'^3$ are the smallest powers of $\wp_3'$ and $\wp_2'$, respectively, for which $[\Sigma_\frak{f}:\Sigma] > 1$, where $\Sigma$ is the Hilbert class field of $K$.  Also, the field $L_{\mathcal{O},9}$ is equal to $\Sigma_{9}\Omega_f$, where $\Sigma_{9}$ is the ray class field over $K$ of conductor $\mathfrak{f} = (9)$, because $d_K \neq -3, -4$ when $-d \equiv 1$ (mod $3$).  \medskip

The connection between this formula and Deuring's formula for the prime $p = 3$ can be looked at as follows.  We know that $\tau_3 \big|_{\Omega_f} = \tau =\left(\frac{\Omega_f/K_d}{\wp_3}\right)$, so if $\tau$ has order $k$, then $\tau_3$ has order $k$ or $3k$, since $\tau_3^{k}$ fixes $\Omega_f$ and $[\Sigma_{\wp_3'^2}\Omega_f: \Omega_f] = 3$.  Thus, the sum \eqref{eqn:1.4} for $n$ is related to the sum \eqref{eqn:1.1} for $n$ and $n/3$.  If $3 \nmid n$, these formulas say that one-third of the time, as weighted by the class number, the automorphisms $\tau_3$ and $\tau$ will have the same order $n$, in which case a prime divisor of $\wp_3$ in $\Omega_f$ will split in $\Sigma_{\wp_3'^2}\Omega_f$.  If $3 \mid n$, the relationship is a little more complicated.  \medskip

As in previous papers \cite{m1, m5, m4, m2, am1}, the class number formula in Theorem \ref{thm:1} is connected to an algebraic function, in that it arises from a formula for the number of periodic points for this algebraic function of primitive (i.e., minimal) period $n$.  There is more to this connection which we will discuss below.  \medskip

The algebraic function $w = \hat F(z)$ that we consider here is defined by its minimal polynomial over $\mathbb{C}(z)$:
$$f(z,w) = w^3 + (-z^3 + 6z^2 - 6z - 1)w^2 + (z^3 - 3z^2 + 3z + 1)w - z^3.$$
Recall that a periodic point of $\hat F(z)$ is a value $a \in \mathbb{C}$ for which there are complex $a_1, a_2, \dots, a_{n-1}$ satisfying the simultaneous equations
$$f(a,a_1) = f(a_1,a_2) = \cdots f(a_{n-1},a) = 0.$$
See \cite{m5}.  There are several main parts of the argument:
\begin{enumerate}[(i)]
\item Defining a polynomial $R_n(x)$ whose roots are all the periodic points of $\hat F(z)$ whose periods divide $n$; computing the degree of $R_n(x)$ and the factorization of $R_n(x)$ modulo $3$ (Propositions \ref{prop:1}, \ref{prop:2}).
\item Determining the fields generated over $\mathbb{Q}$ by periodic points of $\hat F(z)$ (Theorem \ref{thm:2}).  These turn out to be the class fields $\Sigma_{\wp_3'^2}\Omega_f$ over imaginary quadratic fields $K_d = \mathbb{Q}(\sqrt{-d})$, for which $-d \equiv 1$ (mod $3$),  mentioned in Theorem \ref{thm:1}.  The degrees of these class fields over $\mathbb{Q}$ are $6h(-d)$, where $h(-d)$ is the class number of the order $\textsf{R}_{-d}$ of discriminant $-d$ in $K$. This makes use of the solutions of the cubic Fermat equation in the ring class fields $\Omega_f$ which are discussed in \cite{m1}.
\item A $3$-adic argument using the factorization modulo $3$ of $R_n(x)$ to show that $R_n(x)$ has distinct roots (Lemma \ref{lem:2}, Theorem \ref{thm:3}).  This implies that the expression
$$\textsf{P}_n(x) = \prod_{k \mid n}{R_k(x)^{\mu(n/k)}}$$
is a polynomial of degree $2\sum_{k \mid n}{\mu(n/k) 3^k}$.  Part (ii) shows that this polynomial is a product of irreducible polynomials of degree $6h(-d)$ for various discriminants $-d$.
\item Determining the minimal period of a periodic point in terms of the order of the Frobenius automorphism in $\textrm{Gal}(\Sigma_{\wp_3'^2}\Omega_f/K)$ for a prime divisor of $3$ in $K$ (Lemma \ref{lem:1}, Section 5).
\end{enumerate}

Regarding point (iv), there is a $3$-adic branch $F(z)$ of the algebraic function $\hat F(z)$ which is defined and single-valued on a subset of the maximal unramified algebraic extension $\textsf{K}_3$ of the $3$-adic field $\mathbb{Q}_3$, having the property that the periodic points $\xi$ of $\hat F(z)$ in $\textsf{K}_3$ are periodic points in the usual sense of the function $F(z)$.  This function is representable as a convergent $3$-adic series on the domain
$$\textsf{D} = \{z \in \textsf{K}_3 \ | \ z \ \textrm{integral} \ \wedge \ z \not \equiv -1 \ (\textrm{mod} \ 3)\};$$
and $F: \textsf{D} \rightarrow \textsf{D}$ is a lift of the Frobenius automorphism on $\textsf{D}$:
$$F(z) \equiv z^3 \ (\textrm{mod} \ 3), \ z \in \textsf{D}.$$
For those periodic points $\xi$ which lie in $\textsf{D}$, after completing $L \rightarrow L_\frak{p}$ with respect to some prime divisor $\mathfrak{p}$ of $\wp_3$ in $L$, we have the relation
$$\xi^{\tau_3} = F(\xi).$$
This relation makes it clear that the minimal period of $\xi$ with respect to $F(z)$ (and $\hat F(z)$) is equal to the order of the automorphism $\tau_3 \in \textrm{Gal}(L/K_d)$.  This fact leads to the formulation of the set $\mathfrak{D}_{n,3}$ in Theorem \ref{thm:1}.  The same method, using arguments analogous to (i)-(iv), was used to give a new proof for the cases $p = 2, 3, 5$ of Deuring's formula \eqref{eqn:1.1} in \cite{m4, m1, m2}.  \medskip

This led the second author to conjecture the following in \cite{m2}.

\newtheorem{conj}{Conjecture}

\begin{conj}
Let $\mathfrak{f} = \mathfrak{p}'$ be a first degree prime ideal in an imaginary quadratic field  $K_d = \mathbb{Q}(\sqrt{-d})$ in which the prime $(p) = \mathfrak{p} \mathfrak{p}'$ splits, with $-d = d_K f^2$ and $N\mathfrak{p} = p$, where $p=7$ or $11$.  Let $\mathcal{O} = \textsf{R}_{-d}$ and
$$\tau_{\frak{p}} = \left(\frac{F_d/K_d}{\frak{p}}\right),$$
where $F_d$ is the inertia field for $\mathfrak{p}$ in the extended ring class field $L_{\mathcal{O},p} = \Sigma_p \Omega_{pf}$.
If $\mathfrak{D}_{n,p}$ is the set
$$\mathfrak{D}_{n,p} = \Big\{-d \ \bigg| \left(\frac{-d}{p}\right) = +1 \ \wedge \ ord(\tau_{\frak{p}}) = n\Big\},$$
then
\begin{equation}
\sum_{-d \in \frak{D}_{n,p}}{h(-d)} = \frac{2}{\varphi(\mathfrak{f})} \sum_{k \mid n}{\mu(n/k)p^k}, \ \ n > 1.
\label{eqn:1.5}
\end{equation}
\label{conj:1}
\end{conj}

Equation \eqref{eqn:1.2} shows that Conjecture \ref{conj:1} is true for the conductor $\mathfrak{f} = \wp_5'$.  Equations \eqref{eqn:1.3} and \eqref{eqn:1.4} show that the analogue of this conjecture also holds for $\mathfrak{f} = \wp_2'^3$ and $\wp_3'^2$.  We also note that \ref{eqn:1.5} holds for $\mathfrak{f} = \mathfrak{p}' \mid 3$ and $ p = 3$, by the class number formula \eqref{eqn:1.1} of Deuring, since the extended ring class field $L_{\mathcal{O},3}$ coincides with $\Sigma_3 \Omega_f = \Sigma(\rho) \Omega_f = \Omega_f(\rho)$,  where $\rho$ is a primitive cube root of unity.  In this case the inertia field for $\mathfrak{p}$ is just the field $\Omega_f$.  See \cite{d1}, \cite{d2}, \cite{m4}. \medskip

This conjecture was stated overly optimistically in \cite{m2}, for all primes $p > 5$, but in Section 8 we will show that it fails for many primes which are congruent to $11$ (mod $12$).  We also give evidence for the truth of the conjecture for both $p = 7$ and $p = 11$, by showing that it holds for $n = 2, 3, 4$.  See Tables \ref{tab:1}-\ref{tab:4}.  Thus, it seems that the class number formulas discussed here are a phenomenon restricted to small primes (or prime powers), at least for $n \ge 2$.  This leaves open the question of whether these formulas might still hold for other primes for large enough $n$. \medskip

Note that the inertia field $F_d$ in this conjecture is the field $\Sigma_{\frak{p}'} \Omega_f$ in case $d_K \neq -3, -4$.  If $K = \mathbb{Q}(\sqrt{-3})$ or $\mathbb{Q}(\sqrt{-1})$, fields which tend to be exceptional in the theory of complex multiplication, then the inertia field $F_d$ is usually larger than $\Sigma_{\frak{p}'} \Omega_f$.  Furthermore, if $f > 1$ is not divisible by $3$ when $d_K = -3$, or by $2$ when $d_K = -4$, then $F_d$ can be difficult to determine.  (But see Theorem \ref{thm:5}.) This makes the above conjecture difficult to check for primes for which these two fields are possibilities.  However, when these class fields $F_d$ (for a given prime $p$) are generated over $\mathbb{Q}$ by periodic points of an algebraic function (independent of $-d$ and $f$), then this gives a method for determining $F_d$ and exactly which discriminants of the form $-3f^2, -4f^2$ lie in the set $\mathfrak{D}_{n,p}$.  We illustrate how this works for the prime $p = 7$ in Section 8, where such an algebraic function is available.  See Theorem \ref{thm:6}. \medskip

We will give a full proof of this conjecture for $p = 7$ in another paper.

\section{Preliminary results.}
In this section we give some polynomial identities and congruences which will be important for the proof of Theorem \ref{thm:1}. \medskip

Let $f(z,w)$ be the polynomial
\begin{equation}
f(z,w) = w^3 + (-z^3 + 6z^2 - 6z - 1)w^2 + (z^3 - 3z^2 + 3z + 1)w - z^3,
\label{eqn:2.1}
\end{equation}
and let $\hat F(z)$ denote the algebraic function for which $f(z,\hat F(z)) = 0$.  Then one branch of $\hat F(z)$ can be written as
\begin{align}
\notag F(z) &= \frac{1}{3}(z^2 - 4z + 1)(z^3 - 6z^2 + 3z + 1)^{1/3} + \frac{1}{3}(z - 2)(z^3 - 6z^2 + 3z + 1)^{2/3}\\
\label{eqn:2.2} & \ \ + \frac{1}{3}(z^3 - 6z^2 + 6z + 1).
\end{align}

The polynomial $f(z,w)$ is one of the factors of the following resultant.  Let
\begin{align}
\label{eqn:2.3} g(x,y) &= (y^2 + 3y + 9)x^3 - (y + 6)^3,\\
\label{eqn:2.4} h(x,z) &= z^3 - (3 + x)z^2 + zx + 1.
\end{align}
Now form
\begin{align*}
&\textrm{Res}_y(\textrm{Res}_x(g(x,y),h(x,z)),h(y,w)) = \\
& \ \ \ -(w^3z^3 - 2w^2z^3 - 3w^2z^2 + 2wz^3 + 3w^2z - z^3 + w^2 + 3z^2 - w - 3z + 1)\\
& \ \ \ \times (w^3z^3 - 3w^3z^2 - w^2z^3 + 3w^3z + 9w^2z^2 + wz^3 - w^3 - 9w^2z - 6wz^2\\
& \ \ \ \ \ + 2w^2 + 6wz - 2w + 1)\\
& \ \ \ \times [w^3 + (-z^3 + 6z^2 - 6z - 1)w^2 + (z^3 - 3z^2 + 3z + 1)w - z^3]\\
&= -k_1(z,w) k_2(z,w) f(z,w).
\end{align*}
We have the congruences
\begin{align*}
k_1(z,w) & \equiv (wz^3 + 2z^3 + 1)(w + 1)^2 \ (\textrm{mod} \ 3);\\
k_2(z,w) & \equiv (wz^3 + 2w + 1)(w + 1)^2 \ (\textrm{mod} \ 3);\\
f(z,w) & \equiv (2z^3 + w)(w + 1)^2 \ (\textrm{mod} \ 3).
\end{align*}
This yields that
\begin{equation}
k_1(z,z^3) \equiv k_2(z,z^3) \equiv (z+1)^{12} \ \ (\textrm{mod} \ 3),
\label{eqn:2.5}
\end{equation}
and
$$f(z,z^3) \equiv 0 \ (\textrm{mod} \ 3).$$
Furthermore,
\begin{equation}
g\left(\frac{z^3-3z^2+1}{z(z-1)},\frac{w^3-3w^2+1}{w(w-1)}\right) = \frac{-1}{z^3(z-1)^3 w^3(w-1)^3} k_1(z,w) k_2(z,w) f(z,w).
\label{eqn:2.6}
\end{equation}

\section{The polynomials $R_n(x)$.}
In this section we determine a polynomial $R_n(x)$ whose roots are the periodic points of the algebraic function $\hat F(z)$ in $\overline{\mathbb{Q}}$ or $\overline{\mathbb{Q}_3}$ (algebraic closure of the $3$-adic field $\mathbb{Q}_3$). \medskip

As in \cite{m1, m5, am1, am2} we define the resultants $R_n(x)$ for the polynomial $f(z,w)$ inductively, as follows.  We set
\begin{align*}
R^{(1)}(x,x_1) &= f(x,x_1),\\
R^{(n)}(x,x_n) & = \textrm{Res}_{x_{n-1}}(R^{(n-1)}(x,x_{n-1}), f(x_{n-1},x_n)), \ \ n \ge 2.
\end{align*}
Then $R_n(x)$ is obtained from $R^{(n)}(x,x_n)$ by setting $x_n = x$:
$$R_n(x) = R^{(n)}(x,x).$$
It is not hard to see that the roots of $R_n(x)$ are exactly the values $a$ (in some algebraically closed field $F$) for which there are $a_1, \dots, a_{n-1} \in F$ for which the simultaneous equations hold:
$$f(a,a_1) = f(a_1,a_2) = \cdots = f(a_{n-1}, a) = 0.$$
With this definition we will prove the following two propositions.

\newtheorem{prop}{Proposition}

\begin{prop} We have the congruences
\begin{align*}
R^{(n)}(x,x_n) & \equiv -(x^{3^n} -x_n)(x_n+1)^{3^n-1} \ (\textrm{mod} \ 3);\\
R_n(x) & \equiv -(x^{3^n} -x)(x +1)^{3^n-1} \ (\textrm{mod} \ 3).
\end{align*}
\label{prop:1}
\end{prop}

\begin{proof}
Since $f(z,w)=w^3+(-z^3+6z^2-6z-1)w^2+(z^3-3z^2+3z+1)w-z^3$, we find for $n=1$ that
\begin{align*}
R^{(1)}(x,x_1) & = f(x,x_1)\\
& = x_1^3+(-x^3+6x^2-6x-1)x_1^2+(x^3-3x^2+3x+1)x_1-x^3\\
 & \equiv x_1^3+(-x^3-1)x_1^2+(x^3+1)x_1-x^3\ (\textrm{mod}\ 3)\\
 & \equiv-x^3(x_1^2+2x_1+1)+x_1(x_1^2+2x_1+1)\ (\textrm{mod}\ 3)\\
 & \equiv-(x^3-x_1)\,(x_1+1)^2\ (\textrm{mod}\ 3).
\end{align*}
Hence,
$$R_1(x) \equiv -(x^3-x)\,(x+1)^2\ (\textrm{mod}\ 3).$$
Now for the induction step, assume the result is true for $n-1$. Then, modulo $3$, we have
\begin{align*}
R^{(n)}&(x,x_n) = \textrm{Res}_{x_{n-1}}(R^{(n-1)}(x,x_{n-1}),f(x_{n-1},x_n))\\
&\equiv\textrm{Res}_{x_{n-1}}\big(-(x^{3^{n-1}} -x_{n-1})(x_{n-1} +1)^{3^{n-1}-1},-(x_{n-1}^3-x_n)(x_n+1)^2\big).
\end{align*}
The roots of $(x_{n-1}^3-x_n)\,(x_n+1)^2$, as a polynomial in $x_{n-1}$, are $\omega^i\sqrt[3]{x_n}$ for $i=0,1,2$, where $w=e^{2\pi i/3}$. Hence,
using the formulas from \cite[p. 279]{jvu}, we find that
\begin{align*}
\textrm{Res}_{x_{n-1}}&\big(-(x^{3^{n-1}} -x_{n-1})(x_{n-1} + 1)^{3^{n-1}-1},-(x_{n-1}^3-x_n)(x_n+1)^2\big)\\
&=(-1)^{3^{n-1}\cdot3}\big[-(x_n+1)^2\big]^{3^{n-1}} \big[-\big(x^{3^{n-1}}-\sqrt[3]{x_n}\big)\big(\sqrt[3]{x_n}+1\big)^{3^{n-1}-1}\big]\\
& \ \ \ \times\big[-\big(x^{3^{n-1}}-\omega\sqrt[3]{x_n}\big) \big(\omega\sqrt[3]{x_n}+1\big)^{3^{n-1}-1}\big]\\
& \ \ \ \times\big[-\big(x^{3^{n-1}}-\omega^2\sqrt[3]{x_n}\big) \big(\omega^2\sqrt[3]{x_n}+1\big)^{3^{n-1}-1}\big]\\
&=-\big(x_n+1\big)^{2\cdot3^{n-1}} \big(x^{3^n}-x_n\big) \big(x_n+1\big)^{3^{n-1}-1}\\
&=-\big(x^{3^n}-x_n\big) \big(x_n+1\big)^{3^n-1}.
\end{align*}
Hence, we obtain
\begin{align*}
R^{(n)}(x,x_n) & \equiv -(x^{3^n} -x_n)(x_n +1)^{3^n-1} \ (\textrm{mod} \ 3),\\
R_n(x) & \equiv -(x^{3^n} -x)(x +1)^{3^n-1} \ (\textrm{mod} \ 3),
\end{align*}
completing the induction.
\end{proof}

\begin{prop} The polynomial $R^{(n)}(x,x_n)$ has the form\[R^{(n)}(x,x_n)=-(x_n^2-x_n+1)(x_n^3-6x_n^2+3x_n+1)^{3^{n-1}-1}x^{3^n}+S_n(x,x_n),\]where deg$_xS_n(x,x_n)\le3^n-1$ and deg$_{x_n}S_n(x,x_n)=3^n$.  Moreover, the expression containing terms in $S_n(x,x_n)$ that are independent of $x$ is equal to the coefficient of the highest power of $x$ in $R^{(n)}(x,x_n)$ multiplied by $(-x_n)$; in other words, $S_n(0,x_n)$ equals
$$T_n(x_n) = x_n(x_n^2-x_n+1)(x_n^3-6x_n^2+3x_n+1)^{3^{n-1}-1}.$$
Finally, $\textrm{deg}_{x_n}(S_n(x,x_n)-T_n(x_n)) \le 3^n-1$.
\label{prop:2}
\end{prop}

\begin{proof} We prove this by induction. For $n=1$, we have
\begin{align*}
R^{(1)}(x,x_1) &= f(x,x_1)\\
&= -(x_1^2-x_1+1)x^3+(6x_1^2-3x_1)x^2+(-6x_1^2+3x_1)x+(x_1^3-x_1^2+x_1),
\end{align*}
 which satisfies the assertions with
 \begin{align*}
 S_1(x,x_1) &= (6x_1^2-3x_1)x^2+(-6x_1^2+3x_1)x+(x_1^3-x_1^2+x_1),\\
 T_1(x_1) &= x_1^3-x_1^2+x_1.
 \end{align*}
 Proceeding by induction from $n-1$ to $n$, we have
\begin{align*}
R^{(n)}(x,x_n)&=\text{Res}_{x_{n-1}}\big(R^{(n-1)}(x,x_{n-1}),f(x_{n-1},x_n)\big) \\
&=-\text{Res}_{x_{n-1}}\big(f(x_{n-1},x_n),R^{(n-1)}(x,x_{n-1})\big)\\
&=-\big[-(x_n^2-x_n+1)\big]^{3^{n-1}}\prod_i{R^{(n-1)}(x,\rho_i)},
\end{align*}
where $\rho_i,\ 1\le i\le3$, are the roots of $f(x_{n-1},x_n)=0$, as a polynomial in $x_{n-1}$.
Write
\begin{align*}
f(z,w) &= -(w^2-w+1)z^3 +(6w^2-3w)z^2 - (6w^2-3w)z + w^3-w^2+w\\
&= -(w^2-w+1)(z^3-w) +(6w^2-3w)(z^2-z).
\end{align*}
Thus, for $n\ge2$,
\begin{align*}
R^{(n)}&(x,x_n)=(x_n^2-x_n+1)^{3^{n-1}}\\
\times&\prod_i{\left[-(\rho_i^2-\rho_i+1)(\rho_i^3-6\rho_i^2+3\rho_i+1)^{3^{n-2}-1}x^{3^{n-1}}+S_{n-1}(x,\rho_i)\right]}
\end{align*}
We observe that
$$\sum_i\rho_i=\frac{6x_n^2-3x_n}{x_n^2-x_n+1},\ \ \prod_i\rho_i=x_n,\ \ \rho_1\rho_2+\rho_2\rho_3+\rho_3\rho_1=\frac{6x_n^2-3x_n}{x_n^2-x_n+1}.$$
The leading term in $x$ in the above product is
\begin{align*}
-(&x_n^2-x_n+1)^{3^{n-1}}\Big[\prod_i(\rho_i^2-\rho_i+1)(\rho_i^3-6\rho_i^2+3\rho_i+1)^{3^{n-2}-1}\Big]x^{3^n}\\
=-(&x_n^2-x_n+1)^{3^{n-1}}\times\big[(\rho_1^2-\rho_1+1)(\rho_2^2-\rho_2+1)(\rho_3^2-\rho_3+1)\big]\\&\times\big[(\rho_1^3-6\rho_1^2+3\rho_1+1)(\rho_2^3-6\rho_2^2+3\rho_2+1)(\rho_3^3-6\rho_3^2+3\rho_3+1)\big]^{3^{n-2}-1}x^{3^n}.
\end{align*}
Notice that evaluating the product in the first bracketed expression in the above equation is tantamount to finding the resultant 
$$\textrm{Res}_{x_{n-1}}\big(f(x_{n-1},x_n),(x_{n-1}^2-x_{n-1}+1)\big)$$
divided by $(-(x_n^2-x_n+1))^2$. Also note that
\begin{align*}
\textrm{Res}_{x_{n-1}}&\big(f(x_{n-1},x_n),(x_{n-1}^2-x_{n-1}+1)\big)\\
&=(-1)^6\text{Res}_{x_{n-1}}\big((x_{n-1}^2-x_{n-1}+1),f(x_{n-1},x_n)\big).
\end{align*}
The roots of $x_{n-1}^2-x_{n-1}+1$ are $-\omega$ and $-\omega^2$, where $\omega$ is a primitive cube root of unity. Then,
\begin{align*}
\textrm{Res}_{x_{n-1}}&\big((x_{n-1}^2-x_{n-1}+1),f(x_{n-1},x_n)\big)\\
=\big[&(x_n^2-x_n+1)(\omega^3+x_n)+(6x_n^2-3x_n)(\omega^2+\omega)\big]\\
&\times\big[(x_n^2-x_n+1)(\omega^6+x_n)+(6x_n^2-3x_n)(\omega^4+\omega^2)\big]\\
=\big[&(x_n^2-x_n+1)(1+x_n)-(6x_n^2-3x_n)\big]^2\\
=\big(&x_n^3-6x_n^2+3x_n+1\big)^2.
\end{align*}
Hence, we find that
$$(\rho_1^2-\rho_1+1)(\rho_2^2-\rho_2+1)(\rho_3^2-\rho_3+1)=\frac{(x_n^3-6x_n^2+3x_n+1)^2}{(x_n^2-x_n+1)^2}.$$
In a similar way, we calculate the second bracketed expression by finding the resultant $\textrm{Res}_{x_{n-1}}\big(f(x_{n-1},x_n),(x_{n-1}^3-6x_{n-1}^2+3x_{n-1}+1)\big)$ divided by $-(x_n^2-x_n+1)^3$.  Also note that
\begin{align}
\notag \textrm{Res}_{x_{n-1}}\big(&f(x_{n-1},x_n),(x_{n-1}^3-6x_{n-1}^2+3x_{n-1}+1)\big)\\
\label{eqn:5.1} &=-\text{Res}_{x_{n-1}}\big((x_{n-1}^3-6x_{n-1}^2+3x_{n-1}+1),f(x_{n-1},x_n)\big).
\end{align}
We now claim that the roots of $f(x_{n-1},x_n)$ as a polynomial in $x_n$, for $x_{n-1} = \gamma$, a root of $x_{n-1}^3-6x_{n-1}^2+3x_{n-1}+1$, are the roots of the same cubic.  In fact, we have that
\begin{equation}
f(\gamma, x_n) = (x_n-\gamma)^3.
\label{eqn:5.2}
\end{equation}
From this it would follow that the resultant in \eqref{eqn:5.1} is $(x_n^3-6x_n^2+3x_n+1)^3$.  To prove \eqref{eqn:5.2}, note that the roots $\gamma$ of $x^3-6x^2+3x+1$ are fixed points of the function $F(z)$.  Hence $f(\gamma, \gamma) = 0$.  Equation \eqref{eqn:5.2} will follow from the fact that the first two partial derivatives of $f(x_{n-1},x_n)$ with respect to $x_{n}$ vanish at $(x_{n-1},x_n) = (\gamma, \gamma)$.  We first have that
\begin{align*}
\frac{\partial f}{\partial x_n}&=\frac{\partial}{\partial x_n}\big[(x_n^2-x_n+1)(-x_{n-1}^3+x_n)+(6x_n^2-3x_n)(x_{n-1}^2-x_{n-1})\big]\\
&=(2x_n-1)(-x_{n-1}^3+x_n)+(x_n^2-x_n+1)+(12x_n-3)(x_{n-1}^2-x_{n-1}).
\end{align*}
In this expression we substitute $x_{n-1} = x_n = \gamma$ and find that
\begin{align*}
\frac{\partial f}{\partial x_n}\bigg|_{x_n=x_{n-1} = \gamma}&=(2\gamma-1)(-\gamma^3+\gamma)+(\gamma^2-\gamma+1)+(12\gamma-3)(\gamma^2-\gamma)\\
&=-2\gamma^4+13\gamma^3-12\gamma^2+\gamma+1\\
&=(-2\gamma+1)(\gamma^3-6\gamma^2+3\gamma+1)\\
&=0.
\end{align*}
In a similar way, we find that
\begin{equation*}
\frac{\partial^2f}{\partial x_n^2}=2(-x_{n-1}^3+x_n)+(2x_n-1)+(2x_n-1)+12(x_{n-1}^2-x_{n-1})
\end{equation*}
and
\begin{align*}
\frac{\partial^2f}{\partial x_n^2}\bigg|_{x_n=x_{n-1}=\gamma}&=2(-\gamma^3+\gamma)+(2\gamma-1)+(2\gamma-1)+12(\gamma^2-\gamma)\\
&=-2\gamma^3+12\gamma^2-6\gamma-2\\
&=-2(\gamma^3-6\gamma^2+3\gamma+1)\\
&=0.
\end{align*}
Alternatively, we have
$$f(x,y)-(y-x)^3 = -y(y - 1)(x^3 - 6x^2 + 3x + 1),$$
from which \eqref{eqn:5.2} follows directly.  It follows that
$$(\rho_1^3-6\rho_1^2+3\rho_1+1)(\rho_2^3-6\rho_2^2+3\rho_2+1)(\rho_3^3-6\rho_3^2+3\rho_3+1)=-\frac{(x_n^3-6x_n^2+3x_n+1)^3}{(x_n^2-x_n+1)^3}.$$
Hence, the leading term in the expression for $R^{(n)}(x,x_n)$, for $n \ge 2$, is equal to
\begin{align*}
&-(x_n^2-x_n+1)^{3^{n-1}} \frac{(x_n^3-6x_n^2+3x_n+1)^2}{(x_n^2-x_n+1)^2} \left(-\frac{(x_n^3-6x_n^2+3x_n+1)^3}{(x_n^2-x_n+1)^3}\right)^{3^{n-2}-1}x^{3^n}\\
&=-(x_n^2-x_n+1)(x_n^3-6x_n^2+3x_n+1)^{3^{n-1}-1}x^{3^n}.
\end{align*}
The set of terms without $x$ in $R^{(n)}(x,x_n)$ comes from the product of the terms without $x$ in $S_{n-1}(x,\rho_i)$ multiplied by $(x_n^2-x_n+1)^{3^{n-1}}$.  From the inductive hypothesis this product is given by
$$(x_n^2-x_n+1)^{3^{n-1}}\prod_i\big[\rho_i(\rho_i^2-\rho_i+1)(\rho_i^3-6\rho_i^2+3\rho_i+1)^{3^{n-2}-1}\big],$$
and using our previous calculations this simplifies to be
\begin{align*}
\notag T_n(x_n) &= (x_n^2-x_n+1)^{3^{n-1}}(x_n)\frac{(x_n^3-6x_n^2+3x_n+1)^2}{(x_n^2-x_n+1)^2}\\
& \ \ \ \times \left[\frac{(x_n^3-6x_n^2+3x_n+1)^3}{(x_n^2-x_n+1)^3}\right]^{3^{n-2}-1}\\
&=x_n(x_n^2-x_n+1)(x_n^3-6x_n^2+3x_n+1)^{3^{n-1}-1}.
\end{align*}
This shows that
$$R^{(n)}(x,x_n)=-(x_n^2-x_n+1)(x_n^3-6x_n^2+3x_n+1)^{3^{n-1}-1}x^{3^n}+S_n(x,x_n),$$
and it remains to prove the degree assertions for $S_n(x,x_n)$.  After singling out the leading term in the expression for $R^{(n)}(x,x_n)$, which is clearly the highest degree term in $x$, the remaining expression $S_n(x,x_n)$ must have degree in $x$ at most $3^n-1$. To find the degree of $x_n$ in $S_n(x,x_n)$, observe that the largest degree term there comes from the product of the largest degree terms in $S_{n-1}(x,\rho_i)$ multiplied by $(x_n^2-x_n+1)^{3^{n-1}}$.  This is because any term $\rho_1^i \rho_2^j \rho_3^k$ multiplying a nontrivial power $x^r$ combines with a similar term for which $i \ge j \ge k$, and contributes an expression
$$t_r = (x_n^2-x_n+1)^{3^{n-1}} \sigma_1^{i-j} \sigma_2^{j-k} \sigma_3^k = (x_n^2-x_n+1)^{3^{n-1}-(i-k)} (6x_n^2-3x_n)^{i-k} x_n^k$$
to the coefficient of $x^r$ in the product, where the $\sigma_i$ are the elementary symmetric functions of the roots $\rho_i$.  Here $k \le 3^{n-1}-1$, so the degree in $x_n$ of $t_r$ is at most
$$2\cdot 3^{n-1}-2(i-k)+2(i-k)+k \le 3\cdot 3^{n-1}-1 = 3^n-1.$$
Therefore, the leading term in the product is the leading term of $T_n(x_n)$.  This shows that the degree of $x_n$ in $S_n(x,x_n)-T_n(x_n)$ is less than or equal to $3^n-1$, finishing the proof.
\end{proof}

\newtheorem{cor}{Corollary}

\begin{cor}The degree of $R_n(x)$ is $\textrm{deg}(R_n(x)) = 2\cdot3^n-1$.
\label{cor:1}
\end{cor}

\begin{proof} This follows from Proposition \ref{prop:2} on setting $x_n = x$.
\end{proof}

\noindent {\bf Examples.}  From the above definition of $R_n(z)$ we have
\begin{equation}
R_1(z) = f(z,z) = -z(z - 1)(z^3 - 6z^2 + 3z + 1)
\label{eqn:3.1}
\end{equation}
and
\begin{align*}
R_2(z) &= -z(z - 1)(z^3 - 6z^2 + 3z + 1) (z^{12} - 24z^{11} + 202z^{10} - 712z^9 + 1561z^8\\
& \ \ \   - 2308z^7 + 2354z^6 - 1660z^5 + 778z^4 - 208z^3 + 4z^2 + 12z + 1).
\end{align*}
The cubic factor of $R_1(z)$ has discriminant $3^6$ and has the real subfield of the field of $9$-th roots of unity as its splitting field. \medskip

The $12$-th degree factor $p(z)$ of $R_2(z)$ has discriminant $D = 2^{48}3^{30}5^6 7^6$ and factors over $K = \mathbb{Q}(\sqrt{-35})$ as
\begin{align*}
p(z) &= (z^6 - 12z^5 + (4a + 29)z^4 + (-8a - 8)z^3 + (4a - 16)z^2 + 6z+1)\\
& \ \ \ \times ( z^6 - 12z^5 + (-4a + 29)z^4 + (8a - 8)z^3 + (-4a - 16)z^2 + 6z+1),
\end{align*}
where $a = \sqrt{-35}$ and
\begin{align*}
\textrm{disc}_z(z^6 - 12z^5& + (4a + 29)z^4 + (-8a - 8)z^3 + (4a - 16)z^2 + 6z+1)\\
 & = 2^6(1959980a + 8452093) = -2^6\left(\frac{1-\sqrt{-35}}{2}\right)^{15}.
 \end{align*}
In fact, this sextic factors over the Hilbert class field $\Sigma = \mathbb{Q}(\sqrt{5},\sqrt{-7})$ of $K$ as
\begin{align*}
z^6 - &12z^5 + (4a + 29)z^4 + (-8a - 8)z^3 + (4a - 16)z^2 + 6z+1\\
& = \frac{1}{25} (5z^3 + (ab - 10b - 30)z^2 + (-ab + 10b + 15)z+5)\\
& \ \ \ \times (5z^3 + (-ab + 10b - 30)z^2 + (ab - 10b + 15)z+5),
\end{align*}
where $b = \sqrt{5}$.  The discriminant of the first cubic in this factorization is
\begin{align*}
\textrm{disc}&\left(\frac{1}{5}(5z^3 + (ab - 10b - 30)z^2 + (-ab + 10b + 15)z+5)\right)\\
 &= (-288b - 644)a + 936b + 2093,
 \end{align*}
where
$$N_K((-288b - 644)a + 936b + 2093) = -104a-391 = \left(\frac{1-\sqrt{-35}}{2}\right)^6.$$

\begin{prop}
A root of the first factor of the polynomial $p(z)$ generates the ray class field $\Sigma_{\wp_3'^2}$ over $K= \mathbb{Q}(\sqrt{-35})$, where $\wp_3'^2 = \left(\frac{1-\sqrt{-35}}{2}\right)$.
\label{prop:3}
\end{prop}

This fact is a corollary of Theorem \ref{thm:2}, which is proved in the next section.

\section{Fields generated by the periodic points.}

To determine the number fields which are generated by the periodic points of $\hat F(z)$, we adapt the following argument from \cite{m3}. \medskip

Consider the elliptic curve in Deuring normal form
$$E_3(\alpha): \ Y^2+\alpha XY + Y = X^3, \ \ 3(0,0) = O \ (\textrm{base point}),$$
whose $j$-invariant is
$$j(E_3) = \frac{\alpha^3(\alpha^3-24)^3}{\alpha^3-27}.$$
By Proposition 3.6(ii) of \cite{m0} and the remark thereafter, a point $P=(\xi,\eta)$ on $E_3(\alpha)$ satisfies $3P = \pm(0,0)$ whenever its $X$-coordinate satisfies
$$k(x) = h(\alpha, x) = x^3 -(3+\alpha)x^2 +\alpha x+1 = 0.$$
This equation implies the relation 
$$\alpha = \frac{\xi^3-3\xi^2+1}{\xi(\xi-1)}.$$
Since the point $(0,0)$ has order $3$, the point $P= (\xi,\eta)$ is a point of order $9$ on $E_3(\alpha)$. \medskip

Now let $(\alpha, \beta)$ be a solution of
$$\textrm{Fer}_3: \ 27X^3+27Y^3 = X^3 Y^3, $$
as in \cite[Thm. 4.2]{m1}.  Then $\alpha \in \Omega_f$, the ring class field of ring conductor $f$ over the imaginary quadratic field
$$K = \mathbb{Q}(\sqrt{-d}), \ \ -d = d_K f^2 \equiv 1 \ (\textrm{mod} \ 3).$$
Let $(3) = \wp_3 \wp_3'$ in the ring of integers $R_K$ of $K$.  Then $\Omega_f = \mathbb{Q}(\alpha)$ and
$$\alpha-3 \cong \wp_3'^3, \ \ (\alpha) = \wp_3' \mathfrak{a}, \ \textrm{with} \ (\mathfrak{a},3) = 1.$$
(We use Hasse's notation $\cong$ to denote equality of divisors.)  In addition, $\beta = \frac{3(\alpha^\tau+6)}{\alpha-3}$ is a conjugate of $\alpha$ over $\mathbb{Q}$, where $\tau$ is the automorphism
$$\tau = \left(\frac{\Omega_f/K}{\wp_3}\right) \in \textrm{Gal}(\Omega_f/K).$$
We have the discriminant formula
$$\textrm{disc}(x^3 -(3+\alpha)x^2 +\alpha x+1) = (\alpha^2+3\alpha+9)^2.$$
Hence, this discriminant is relatively prime to $\wp_3$.  Furthermore,
$$(\alpha^2+3\alpha+9)(\alpha-3) = \alpha^3-27 = \frac{27\alpha^3}{\beta^3} \cong \wp_3'^6 \mathfrak{c},$$
for some integral ideal $\mathfrak{c}$ prime to $(3)$.  In fact $\mathfrak{c} = (1)$, since
$$(\alpha^3-27)(\beta^3-27) = \alpha^3 \beta^3-27\alpha^3-27\beta^3 + 27^2 = 3^6.$$
Thus, it is clear that $(\alpha^2+3\alpha+9) \cong \wp_3'^3$ and the discriminant of the polynomial $k(x)$ over $\Omega_f$ has the divisor $\wp_3'^6$. \medskip

\newtheorem{lem}{Lemma}

\begin{lem}
With $\alpha$ and $\Omega_f$ as above, the polynomial $k(x) = x^3 -(3+\alpha)x^2 +\alpha x+1$ is irreducible over $\Omega_f$.
\label{lem:0}
\end{lem}

\begin{proof}
This may be shown using the Newton polygon for the shifted polynomial
\begin{equation}
\label{eqn:4.1} k\left(x+\frac{\alpha}{3}+1\right) = x^3 -\frac{\alpha^2 +3\alpha +9}{3}x -\frac{(2\alpha + 3)(\alpha^2 + 3\alpha + 9)}{27},
\end{equation}
for a prime divisor $\mathfrak{p}$ of $\wp_3'$ in $\Omega_f$.  I claim that the additive valuation $w_\mathfrak{p}$ of the last two coefficients is $2$.  For the coefficient of $x$, this follows from the above remarks, since
$$w_\mathfrak{p}\left(\frac{\alpha^2 +3\alpha +9}{3}\right) = 3-1 = 2.$$
For the constant term, use that $\alpha = 3+\gamma^3$, with $\gamma \cong \wp_3'$, by \cite[Thm. 3.4]{m1}, to obtain
$$w_\mathfrak{p}\left(\frac{(2\alpha + 3)(\alpha^2 +3\alpha +9)}{27}\right) = w_\mathfrak{p}(2\alpha + 3) = w_\mathfrak{p}(9+2\gamma^3) = 2.$$
It follows that the Newton polygon for the polynomial in \eqref{eqn:4.1} is the line segment joining the points $(0,2)$ and $(3,0)$, since $(1,2)$ and $(2,\infty)$ lie above this line segment.  The slope of this segment is $-2/3$, which implies the irreducibility of $k(x)$ over the completion $\Omega_{f,\mathfrak{p}}$.  (See \cite[pp. 80-81]{vdw} or \cite[pp. 87-88]{vdw1}.)
\end{proof}

The irreducible polynomial $k(x) = x^3 -(3+\alpha)x^2 +\alpha x+1$ has a square discriminant over $\Omega_f$, so its root $\xi$ generates a cyclic cubic extension of $\Omega_f$, and its conjugates over $\Omega_f$ are $\frac{1}{1-\xi}$ and $\frac{\xi-1}{\xi}$.  Also, since the invariants $g_2, g_3$ and $\Delta$ for the curve $E_3(\alpha)$ are
\begin{align*}
g_2 = \frac{1}{12}(\alpha^4-&24\alpha), \ g_3 = \frac{-1}{216}(\alpha^6-36\alpha^3+216),\\
& \Delta = g_2^3-27g_3^2 = \alpha^3-27,
\end{align*}
they lie in $\Omega_f$ (see \cite[p. 18]{m3}).  Now the fact that the Weber function \cite[p. 135]{si}
\begin{align*}
h(\xi,\eta) & = -2^7 3^5 \frac{g_2 g_3}{\Delta}\left(\xi+\frac{\alpha^2}{12}\right)\\
& = \frac{\alpha (\alpha^3-24)(\alpha^6-36\alpha^3+216)}{\alpha^3-27} (\alpha^2+12\xi)
\end{align*}
lies in $\Sigma_{9}\Omega_f$ (see \cite[p. 1978]{lm} or \cite[Satz 2]{fr}) implies that $L = \Omega_f(\xi) \subset \Sigma_{9}\Omega_f$.  Since $\wp_3'^2 || 9$ and $\Sigma_{\wp_3'} = \Sigma$, it follows that the $3$-part of the conductor of $L/K$ is $\wp_3'^2$.  \medskip

The field $\Omega_f$ has conductor $(f_0)$ over $K = \mathbb{Q}(\sqrt{-d})$, where $f_0 = f$ unless $f = 2$ and $K = \mathbb{Q}(i)$ or $\mathbb{Q}(\sqrt{-3})$, in which case $f_0 = 1$; or $f = 2f'$, with $f'$ odd and $-d/4 \equiv 1$ (mod $8$), in which case $f_0 = f/2$. See Cox \cite[p. 177]{co}.  By the previous paragraph, the field $L$, which is the compositum of $\Omega_f$ and $\Sigma(\xi)$, has conductor $\mathfrak{f} = \wp_3'^2(f_0)$ over $K$, and $\wp_3$ is unramified in $L/K$.  This shows that $L = \Sigma_{\wp_3'^2}\Omega_f$, since $\frac{\varphi_K(\wp_3'^2)}{2} = 3$.  This also shows that $L$ is the inertia field for the prime $\wp_3$ in $\Sigma_{9}\Omega_f/K$, since any intermediate field of $\Sigma_{9}\Omega_f/\Omega_f$ not contained in $L$ must have a conductor which is divisible by $\wp_3$.  \medskip

Now let $\xi$ be a periodic point of $\hat F(z)$ with period $n > 1$, so that $\xi \neq 0, 1$.  Then there are $\xi_0 = \xi, \xi_1, \dots, \xi_{n-1}$ for which
$$f(\xi, \xi_1) = f(\xi_1,\xi_2) = \cdots f(\xi_{n-1},\xi) = 0.$$
Let
$$\alpha_i = \frac{\xi_i^3-3\xi_i^2+1}{\xi_i(\xi_i-1)}, \ 0 \le i \le n-1.$$
Then $f(\xi_i,\xi_{i+1}) = 0$ and the formula \eqref{eqn:2.6} imply that
$$g( \alpha_i, \alpha_{i+1}) = 0, \ 0 \le i \le n-1, \ \alpha_n = \alpha_0.$$
This shows that $\alpha = \alpha_0 = \frac{\xi^3-3\xi^2+1}{\xi(\xi-1)}$ is a periodic point for the algebraic function defined by $g(x,y) = 0$.  Furthermore, $\alpha \neq 3$, since
$$\alpha - 3 = \frac{\xi^3-3\xi^2+1}{\xi(\xi-1)} - 3 = \frac{\xi^3-6\xi^2+3\xi+1}{\xi(\xi-1)}$$
and the roots of $x^3-6x^2+x+1$ have period $1$, by \eqref{eqn:3.1}.  Theorem 2 of \cite{m1} implies that $\alpha$ is the root of some polynomial $p_d(x)$, for which $-d =  d_kf^2 \equiv 1$ (mod $3$) and $\mathbb{Q}(\alpha) = \Omega_f$ is the ring class field for the field $K = \mathbb{Q}(\sqrt{-d})$.  Now the above calculation shows that $\Omega_f(\xi) = \Sigma_{\wp_3'^2}\Omega_f$, if $(\alpha-3) = \wp_3'^3$; and $\Omega_f(\xi) = \Sigma_{\wp_3^2}\Omega_f$, if $(\alpha-3) = \wp_3^3$.  Hence, every periodic point of  $\hat F(z)$ generates a field of the form $\Sigma_{\wp_3'^2}\Omega_f$ or $\Sigma_{\wp_3^2}\Omega_f$ over some quadratic field $K$.  \medskip

Equation \eqref{eqn:2.6} also implies the converse, namely, that every class field of the form $\Sigma_{\wp_3'^2}\Omega_f$ over a quadratic field $K = \mathbb{Q}(\sqrt{-d})$ is generated by a periodic point of $\hat F(z)$.  To show this, let $\alpha$ be a generator of the ring class field $\Omega_f/K$ over $\mathbb{Q}$, as before; and let $\tau_3 = \left(\frac{\Sigma_{\wp_3'^2}\Omega_f/K}{\wp_3}\right)$, so that
$$\tau_3|_{\Omega_f} = \tau =  \left(\frac{\Omega_f/K}{\wp_3}\right).$$
By the results of \cite[Sec. 4]{m1}, we know that
$$g(\alpha,\alpha^\tau) = 0.$$
If $\xi$ is a root of $x^3-(3+\alpha)x^2+\alpha x+1 = 0$, \eqref{eqn:2.6} implies that
\begin{equation}
0 = g(\alpha,\alpha^\tau) = \frac{-1}{\xi(\xi-1)\xi^{\tau_3}(\xi^{\tau_3}-1)} k_1(\xi, \xi^{\tau_3}) k_2(\xi, \xi^{\tau_3}) f(\xi, \xi^{\tau_3}).
\label{eqn:4.2}
\end{equation}
But the congruences \eqref{eqn:2.5} for $k_1, k_2$ yield that
$$k_i(\xi,\xi^{\tau_3}) \equiv k_i(\xi, \xi^3) \equiv (\xi+1)^{12} \ (\textrm{mod} \ \wp_3), \ \ i = 1,2.$$
If $\xi \equiv -1$ modulo some prime divisor $\mathfrak{q}$ of $\wp_3$ in $\Sigma_{\wp_3'^2} \Omega_f$, then
$$0 = \xi^3-(3+\alpha)\xi^2+\alpha \xi+1 \equiv -1-(3+\alpha)-\alpha+1 = -3 -2\alpha \ (\textrm{mod} \ \mathfrak{q}).$$
But this implies that $\alpha \equiv 0$ (mod $\mathfrak{q}$), which is impossible, since $(\alpha,\wp_3) = 1$.  Hence, $k_i(\xi,\xi^{\tau_3}) \not \equiv 0$ (mod $\wp_3$), for $i = 1,2$, implying that $k_i(\xi,\xi^{\tau_3}) \neq 0$.  Therefore, \eqref{eqn:4.2} implies that $f(\xi,\xi^{\tau_3}) = 0$.  This gives easily that $\xi$ is a periodic point of $\hat F(z)$. \medskip

We summarize this discussion as follows.

\begin{thm} (a) Every periodic point of the algebraic function $\hat F(z)$ with period $n > 1$ generates a class field of the form $\Sigma_{\wp_3'^2}\Omega_f$ or $\Sigma_{\wp_3^2}\Omega_f$ over a quadratic field $K = \mathbb{Q}(\sqrt{-d})$ in which $-d = d_K f^2$ and $(3) = \wp_3 \wp_3'$ splits. \medskip

(b) Conversely, every class field $\Sigma_{\wp_3'^2}\Omega_f$ over a quadratic field $K = \mathbb{Q}(\sqrt{-d})$, with $-d = d_K f^2 \equiv 1$ (mod $3$), is generated over $\mathbb{Q}$ by a periodic point of $\hat F(z)$. \medskip

(c) The field $\Sigma_{\wp_3'^2}\Omega_f = \mathbb{Q}(\xi)$ is the inertia field for $\wp_3$ in the extended ring class field $\Sigma_9 \Omega_f$.
\label{thm:2}
\end{thm}

As a corollary of the proof, we see that every periodic point $\xi$ with period greater than $1$ satisfies a polynomial equation $x^3-(3+\alpha)x^2+\alpha x+1 = 0$.  But the quantity $\alpha$ is an algebraic integer, so that $\xi$ must be a unit.  This is certainly also true for the nonzero fixed points.

\begin{cor}
The nonzero periodic points of the function $\hat F(z)$ are units in the abelian extension $\Sigma_{\wp_3'^2}\Omega_f$ over some quadratic field $K = \mathbb{Q}(\sqrt{-d})$, where $-d \equiv 1$ (mod $3$) and $(3) = \wp_3 \wp_3'$ in $R_K$.
\label{cor:2}
\end{cor}

\section{The $3$-adic function $F(z)$.}

The function $F(z)$ in \eqref{eqn:2.2} has the following $3$-adic expansion.  Note first that
$$r(z) = z^3-6z^2+3z+1 = (z+1)^3-9z^2.$$
Hence,
$$r(z)^{1/3} = \left((z+1)^3-9z^2\right)^{1/3} = (z+1)\left(1-\frac{9z^2}{(z+1)^3}\right)^{1/3}.$$
Thus, from \eqref{eqn:2.2} we have that
\begin{align*}
F(z) &= \frac{1}{3}(z^2 - 4z + 1)r(z)^{1/3} + \frac{1}{3}(z - 2)r(z)^{2/3} + \frac{1}{3}(z^3 - 6z^2 + 6z + 1)\\
&= \frac{1}{3}(z^2 - 4z + 1)(z+1)\left(1-\frac{9z^2}{(z+1)^3}\right)^{1/3} \\
& \ \ \ + \frac{1}{3}(z - 2)(z+1)^2\left(1-\frac{9z^2}{(z+1)^3}\right)^{2/3} + \frac{1}{3}(z^3 - 6z^2 + 6z + 1);
\end{align*}
and using Newton's binomial series gives
\begin{align}
\label{eqn:5.3} F(z) &= \frac{1}{3}(z^2 - 4z + 1)(z+1) \sum_{n=0}^\infty{\binom{1/3}{n} (-1)^n 3^{2n} \frac{z^{2n}}{(z+1)^{3n}}}\\
\label{eqn:5.4} & \ \ \ + \frac{1}{3}(z - 2)(z+1)^2 \sum_{n=0}^\infty{\binom{2/3}{n} (-1)^n 3^{2n} \frac{z^{2n}}{(z+1)^{3n}}}\\
\notag & \ \ \ + \frac{1}{3}(z^3 - 6z^2 + 6z + 1).
\end{align}
The series in this formula converge for $z \not \equiv -1$ (mod $3$) in $\textsf{K}_3$, the maximal unramified, algebraic extension of the $3$-adic field $\mathbb{Q}_3$, since
$$3^{n+\lfloor n/2\rfloor} \binom{1/3}{n} \in \mathbb{Z}_3, \ \ 3^{n+\lfloor n/2\rfloor} \binom{2/3}{n} \in \mathbb{Z}_3,$$
where $\mathbb{Z}_3$ denotes the ring of $3$-adic integers in $\mathbb{Q}_3$.  See \cite[Eq. (5.2)]{am2}.  Approximating the series in \eqref{eqn:5.3}, \eqref{eqn:5.4} using the first two terms yields that
\begin{align*}
F(z) &\equiv  \frac{1}{3}(z^2 - 4z + 1)(z+1)\left(1-\frac{3z^2}{(z+1)^3}\right) +  \frac{1}{3}(z - 2)(z+1)^2 \left(1-\frac{6z^2}{(z+1)^3}\right)\\
& \ \ \ + \frac{1}{3}(z^3 - 6z^2 + 6z + 1) \ (\textrm{mod} \ 3)\\
&\equiv z^3 - \frac{6z^4}{(z+1)^2} \ (\textrm{mod} \ 3).
\end{align*}
Thus,
\begin{equation}
F(z) \equiv z^3 \ (\textrm{mod} \ 3), \ \ z \ \textrm{integral}, z \not \equiv -1 \ (\textrm{mod} \ 3),
\label{eqn:5.5}
\end{equation}
and $F(z)$ is a lift of the Frobenius automorphism on the set
$$\textsf{D} = \{z \in \mathfrak{o}_3: z \not \equiv -1 \ (\textrm{mod} \ 3)\},$$
where $\mathfrak{o}_3 \subset \textsf{K}_3$ is the valuation ring in $\textsf{K}_3$. \medskip

\begin{lem} 
For $z \in \textsf{D}$, $w = F(z)$ is the only solution of $f(z,w) = 0$ lying in $\textsf{K}_3$.
\label{lem:1}
\end{lem}

\begin{proof}
For $z \in \textsf{D}$, we know that $w=F(z)$ is one of the roots in $\textsf{K}_3$ of $f(z,w) = 0$.  If all three roots of this equation were in $\textsf{K}_3$, then the discriminant of the cubic $f(z,w)$ would be a square in $\textsf{K}_3$.  But this discriminant is
$$D = \textrm{disc}_w(f(z,w)) = -3(z^3 - 3z^2 + 1)^2 (z^3 - 6z^2 + 3z + 1)^2.$$
A root of either polynomial $z^3-3z^2+1$ or $z^3 - 6z^2 + 3z + 1$ generates the real subfield of the field of $9$-th roots of unity over $\mathbb{Q}$, which is totally ramified at $p = 3$, so these polynomials have no roots in $\textsf{K}_3$.  It follows that $D \neq 0$ is not a square in $\textsf{K}_3$, since $\sqrt{-3} \notin \textsf{K}_3$.  This proves the lemma.
\end{proof}

From the proof of Theorem \ref{thm:2}, we know that for any periodic point $\xi$ satisfying $\alpha = \frac{\xi^3-3\xi^2+1}{\xi^2-\xi}$, we have
$$f(\xi, \xi^{\tau_3}) = 0.$$
Since $\xi \in L = \Sigma_{\wp_3'^2} \Omega_f$, $\xi$ lies in the completion $L_\frak{p} \subset \textsf{K}_3$, for any prime divisor $\mathfrak{p}$ of $\wp_3$, since $\mathfrak{p}$ is not ramified in $L/\mathbb{Q}$.  Now $\xi \not \equiv -1$ implies by Lemma \ref{lem:1} that
$$\xi^{\tau_3} = F(\xi).$$
Furthermore, $F(\xi) \equiv \xi^3\not \equiv -1$ (mod $3$) in $\textsf{K}_3$, since $\xi^3+1 \equiv (\xi+1)^3$ (mod $3$).  Thus $F: \textsf{D} \rightarrow \textsf{D}$ and
$$\xi^{\tau_3^n} = F^n(\xi) \ \textrm{in} \ \textsf{K}_3.$$
This implies that $n = \textrm{ord}(\tau_3)$ is the minimal period of $\xi$ with respect to the map $F(z)$, since $\mathbb{Q}(\xi) = L$.  This implies further that the minimal period of $\xi$ with respect to the algebraic function $\hat F(z)$ is also $n = \textrm{ord}(\tau_3)$, since the branch $F(z)$ is single-valued on $\textsf{K}_3$.

\begin{lem}
If $\gamma \in \mathbb{Q}(\zeta_9)$ is a root of $x^3-6x^2+3x+1$, then
\begin{equation}
(z-\gamma)^3 (w-\gamma)^3f\left(\frac{\gamma z+1-\gamma}{z-\gamma},\frac{\gamma w+1-\gamma}{w-\gamma}\right) = 3^3(22\gamma^2-13\gamma-4)f(w,z).
\label{eqn:5.6}
\end{equation}
\label{lem:2}
\end{lem}

\begin{proof}
This follows by a straightforward calculation.
\end{proof}

\noindent {\bf Remark.} Note that $\varepsilon= 22\gamma^2-13\gamma-4 = \gamma^2(\gamma-1)^2$ is a unit in $\mathbb{Q}(\gamma) \subset \mathbb{Q}(\zeta_9)$, where $\zeta_9 = e^{2\pi i/9}$.  Also, we have, in terms of $\zeta = \zeta_9$, that the roots of $x^3-6x^2+3x+1$ are
\begin{align*}
\gamma_1 &= \zeta^5 + 2\zeta^4 - \zeta^2 +  \zeta + 2,\\
\gamma_2 &= \frac{1}{1-\gamma_1} = -\gamma_1^2 + 5\gamma_1 + 2 = \zeta^5 - \zeta^4 + 2\zeta^2 - 2\zeta + 2,\\
\gamma_3 &= \frac{\gamma_1-1}{\gamma_1} = \gamma_1^2 - 6\gamma_1 + 4 = -2\zeta^5 - \zeta^4 - \zeta^2 + \zeta + 2.
\end{align*}
\medskip

Let $\sigma, \psi_i$ be the mappings
$$\sigma(z) = \frac{1}{1-z}, \ \ \psi_i(z) = \frac{\gamma_i z+1-\gamma_i}{z-\gamma_i}, \ 1 \le i \le 3.$$
It is not hard to check the following:
\begin{align*}
\psi_1 \circ \sigma(z) &= \frac{\sigma^2(\gamma_1) z+1-\sigma^2(\gamma_1)}{z-\sigma^2(\gamma_1)} = \psi_3(z),\\
\psi_1 \circ \sigma^2(z) &= \frac{\sigma(\gamma_1) z+1-\sigma(\gamma_1)}{z-\sigma(\gamma_1)} = \psi_2(z),\\
\psi_1 \circ \psi_2(z) &= \sigma^2(z), \ \ \psi_2 \circ \psi_1(z) = \sigma(z),\\
\psi_1 \circ \psi_3(z) &= \sigma(z), \ \ \psi_3 \circ \psi_1(z) = \sigma^2(z).
\end{align*}
Here we have used that the roots of $x^3-6x^2+3x+1$ are $\gamma_1, \gamma_2 = \sigma(\gamma_1), \gamma_3 = \sigma^2(\gamma_1)$.  Thus,
$\sigma^3 = 1, \psi_i^2 = 1, \psi_1 \sigma \psi_1 = \sigma^2$, so that
$$\{1, \sigma, \sigma^2, \psi_1, \psi_2, \psi_3\} \cong S_3.$$

\section{Proof of Theorem 1.}

Our main result will follow from Theorem \ref{thm:2}, Proposition \ref{prop:1}, and the Corollary to Proposition \ref{prop:2}.  Let $L$ be the field $L = \Sigma_{\wp_3'^2}\Omega_f$ discussed in Theorem \ref{thm:1}.  Then $[L:\mathbb{Q}]=6h(-d)$ and $L$ is the inertia field for $\wp_3$ in the field $\Sigma_9\Omega_f$, an extended ring class field over $K_d = \mathbb{Q}(\sqrt{-d})$.  As in Section 3, let $\displaystyle \tau_3 = \left(\frac{L/K_d}{\wp_3}\right)$ be the Artin symbol for $\wp_3$ in the extension $L/K_d$.  Now define the set of discriminants
\begin{equation}
\mathfrak{D}_{n,3} = \{-d < 0 \ |  -d \equiv 1 \ (\textrm{mod} \ 3) \ \textrm{and} \ \textrm{ord}(\tau_3) = n \ \textrm{in} \ \textrm{Gal}(L/K_d)\}.
\label{eqn:6.1}
\end{equation}

\begin{thm} If $n \ge 2$, we have the following relation between class numbers of discriminants in the set $\mathfrak{D}_{n,3}$:
\begin{equation}
\sum_{-d \in \mathfrak{D}_{n,3}}{h(-d)} = \frac{1}{3} \sum_{k \mid n}{\mu(n/k) 3^k}.
\label{eqn:6.2}
\end{equation}
\label{thm:3}
\end{thm}

\begin{proof}
This proof mirrors the arguments in \cite[pp.822-823]{am1}.  First, define
\begin{equation}
\textsf{P}_n(x) = \prod_{k \mid n}{R_k(x)^{\mu(n/k)}}.
\label{eqn:6.3}
\end{equation}
We show that $\textsf{P}_n(x) \in \mathbb{Z}[x]$.  From Proposition \ref{prop:1} it is clear that $R_n(x)$, for $n > 1$, is divisible (mod $3$) by the $N$ irreducible (monic) polynomials $\bar h_i(x)$ of degree $n$ over $\mathbb{F}_3$,
where
$$N= \frac{1}{n}\sum_{k \mid n}{\mu(n/k)3^k},$$
and that these polynomials are simple factors of $R_n(x)$ (mod $3$).  It follows from Hensel's Lemma that $R_n(x)$ is divisible by distinct irreducible polynomials $h_i(x)$ of degree $n$ over $\mathbb{Z}_3$, the ring of integers in $\mathbb{Q}_3$, for $1 \le i \le N$, with $h_i(x) \equiv \bar h_i(x)$ (mod $3$). In addition, all the roots of $h_i(x)$ are periodic of minimal period $n$ and lie in the unramified extension $\textsf{K}_3$.  Furthermore, $n$ is the smallest index for which $h_i(x) \mid R_n(x)$ over $\mathbb{Q}_3$. \medskip

Now consider the identity
\begin{equation}
(x-\gamma)^2(y-\gamma)^2 f(A(x),A(y)) = 3^3 \gamma^2(\gamma-1)^2 f(y,x),
\label{eqn:6.4}
\end{equation}
where $\displaystyle A(x) = \frac{\gamma x+1-\gamma}{x-\gamma}$, from Lemma \ref{lem:2}.  Note that $\gamma+1$ is a root of the polynomial
$$(x-1)^3-6(x-1)^2+3(x-1)+1 = x^3 - 9x^2 + 18x - 9;$$
substituting $x = \pi^2$ in this polynomial yields that
$$0 = \pi^6 - 9\pi^4 + 18\pi^2 - 9 = (\pi^3 - 3\pi^2 + 3)(\pi^3 + 3\pi^2 - 3).$$
We take $\pi$ to be a root of $x^3 + 3x^2 - 3 = 0$.  Thus, $\pi$ is a prime element in $\textsf{K}_3(\pi) = \textsf{K}_3(\gamma)$; and $\pi^3 \cong 3$ and $\pi^2 = \gamma+1$. \medskip

If the periodic point $a$ of $\hat F(x)$, with minimal period $n > 1$, is a root of one of the polynomials $h_i(x)$, then $a$ is a unit in $\textsf{K}_3$, and for some $a_1, \dots, a_{n-1}$ we have
\begin{equation}
f(a,a_1)=f(a_1,a_2) = \cdots = f(a_{n-1},a)=0.
\label{eqn:6.5}
\end{equation}
Now, $a \not \equiv -1 \ (\textrm{mod} \ \pi)$, since otherwise its reduction $a \equiv \bar a \equiv 1$ (mod $3$) would have degree $1$ over $\mathbb{F}_3$ (using that $\textsf{K}_3$ is unramified over $\mathbb{Q}_3$).  Hence, $a+1$ is a unit in $\textsf{K}_3(\gamma)$.  All of the $a_i$ satisfy $a_i \not \equiv -1$ (mod $\pi$), since the congruence $f(-1,y) \equiv (y+1)^3$ (mod $3$) has only $y \equiv -1$ as a solution.  Hence, if some $a_i \equiv -1$ (mod $\pi$), then $a_j \equiv -1$ for $j >i$, which would imply that $a \equiv -1$ (mod $\pi$), as well.  The elements $b_i=A(a_i)$ are distinct and lie in $\textsf{K}_3(\gamma)$ and satisfy
$$b_i +1 \equiv \frac{(\gamma+1) a_i+1-2\gamma}{a_i-\gamma} = \frac{\pi^2 a_i + 3 - 2\pi^2}{a_i - \gamma} \equiv 0 \ (\textrm{mod} \ \pi),$$
since $a_i - \gamma \equiv a_i + 1 \not \equiv 0$ (mod $\pi$).  The identity (\ref{eqn:6.4}) yields that
\begin{equation}
f(b,b_{n-1})=f(b_{n-1},b_{n-2})= \dots = f(b_1,b)=0
\label{eqn:6.6}
\end{equation}
in $\textsf{K}_3(\gamma)$.  Hence, $b_i \equiv -1$ (mod $\pi$), and the orbit $\{b,b_{n-1},\dots,b_1\}$ is distinct from all the orbits in (\ref{eqn:6.5}). \medskip

Now the map $A(x)$ has order $2$, so it is clear that $b=A(a)$ has minimal period $n$ in (\ref{eqn:6.6}), since otherwise $a=A(b)$ would have period smaller than $n$.  It follows that there are at least $2N$ periodic orbits of minimal period $n>1$.  Noting that
$$R_1(x)= f(x,x) = \ -x(x - 1)(x^3 - 6x^2 + 3x + 1),$$
these distinct orbits and factors account for at least
\begin{equation*}
5 + \sum_{d \mid n,d>1}({2 \sum_{k \mid d}{\mu(d/k)3^k})} = -1+ 2\sum_{d \mid n}({\sum_{k \mid d}{\mu(d/k)3^k})} = 2 \cdot 3^n-1
\end{equation*}
roots, and therefore all the roots, of $R_n(x)$.  This shows that the roots of $R_n(x)$ are distinct and the expressions $\textsf{P}_n(x)$ are polynomials.  Furthermore, over $\textsf{K}_3(\gamma)$ we have the factorization
\begin{equation}
\textsf{P}_n(x) = \pm \prod_{1 \le i \le N}{h_i(x) \tilde h_i(x)}, \ \ n>1,
\label{eqn:6.7}
\end{equation}
where $\tilde h_i(x) = c_i(x-\gamma)^{n}h_i(A(x))$, and the constant $c_i$ is chosen to make $\tilde h_i(x)$ monic. \medskip

For each discriminant $-d \in \mathfrak{D}_{n,3}$, let $f_d(x)$ denote the minimal polynomial of a periodic point $\xi$ of $\hat F(z)$ which generates $\Sigma_{\wp_3'^2} \Omega_f$ over $\mathbb{Q}$ and for which $\alpha = \frac{\xi^3-3\xi^2+1}{\xi(\xi-1)}$ is a root of $p_d(x)$.  Then $f_d(x) \mid \textsf{P}_n(x)$.  Furthermore, every root of $\textsf{P}_n(x)$ is a root of some $f_d(x)$, by Theorem \ref{thm:2}, where $ord(\tau_3) = n$ in order for the roots of $f_d(x)$ to have minimal period $n$.  Finally, for each discriminant $-d \equiv 1$ (mod $3$) there is only one polynomial $f_d(x)$, since there is only one polynomial $p_d(x)$ having $\alpha$ as a root, for each discriminant, by the proof of \cite[Thm. 4.2, p. 879]{m1}.  It follows that
\begin{equation*}
\textsf{P}_n(x) = \pm \prod_{-d \in \mathfrak{D}_{n,3}}{f_d(x)}.
\end{equation*}
Taking degrees on both sides and using (\ref{eqn:6.3}) and $\textrm{deg}(f_d(x)) = [L:\mathbb{Q}] = 6h(-d)$ gives the formula
$$2\sum_{k \mid n}{\mu(n/k) 3^k} = \sum_{-d \in \mathfrak{D}_{n,3}}{6h(-d)}.$$
The formula of the theorem follows.
\end{proof}

The result of Theorem \ref{thm:1} is the analogue of \cite[Thm.1.3]{m2} for the prime $3$ in place of $5$.  The factor $1/3$ in front is to be interpreted as $2/\phi(9)$, replacing the factor $2/\phi(5)$ in the result of \cite{m2}. \medskip

\section{More examples.}

\begin{thm}
The period $n \ge 2$ which corresponds to a given discriminant $-d \equiv 1$ (mod $3$) is the smallest positive integer for which
$$\wp_3^n = \left(\frac{x+y\sqrt{-d}}{2}\right), \ \frac{x+y\sqrt{-d}}{2} \in \mathcal{O} = \textsf{R}_{-d}, \ \textrm{with} \ x \equiv \pm 1 \ (\textrm{mod} \ 9),$$
and $\alpha = \frac{x+y\sqrt{-d}}{2} \equiv \pm 1$ (mod $\wp_3'^2$).
\label{thm:4}
\end{thm}

\begin{proof}
If $\tau_3$ has order $n \ge 2$ in $\textrm{Gal}(\Sigma_{\wp_3'^2} \Omega_f/K)$, then $\tau_3^n$ fixes $\Omega_f$, which implies that $\wp_3^n$ is a principal ideal in the ring $\textsf{R}_{-d}$, by the Artin correspondence.  Furthermore, 
$$\wp_3'^n = \left(\frac{x-y\sqrt{-d}}{2}\right),$$
which implies that
$$\sqrt{-d} \equiv \frac{x}{y} \ (\textrm{mod} \ \wp_3'^2).$$
Hence,
$$\alpha=\frac{x+y\sqrt{-d}}{2} \equiv \frac{2x}{2} \equiv x \ (\textrm{mod} \ \wp_3'^2).$$
It follows that $x \equiv \pm 1$ (mod $\wp_3'^2$) and therefore $x \equiv \pm 1$ (mod $9$), since $(\alpha)$ lies in the ideal group corresponding to the class field $\Sigma_{\wp_3'^2}/K$.  The same relations show that $n$ must be the smallest positive integer with this property, since they imply that $\tau_3^n = 1$.
\end{proof}

\noindent {\bf Remark.} In $K = \mathbb{Q}(\sqrt{-11})$ we have $\wp_3 = \left(\frac{1+\sqrt{-11}}{2}\right) = (\alpha)$, but $\alpha \not \equiv \pm 1$ (mod $\wp_3'^2$), since
$$\alpha - 1 = \frac{-1+\sqrt{-11}}{2} \cong \wp_3', \ \ \alpha + 1 = \frac{3 + \sqrt{-11}}{2} \cong \mathfrak{p}_5.$$
Thus, the period corresponding to $-d = -11$ is not $1$, but $n = 3$, since the ray class field of conductor $\wp_3'^2$ is generated over $K$ by a root of the polynomial
$$p(x) = z^6 - 4z^5 + 13z^4 - 16z^3 + 8z^2 - 2z + 1,$$
which is a primitive factor of $R_3(x)$ (it does not divide $R_1(x)$).  There are no class fields of quadratic fields corresponding to fixed points of $\hat F(z)$, since a root of $z^3-6z^2+3z+1$ generates the cyclic cubic extension $\mathbb{Q}(\zeta_9)^+$ of $\mathbb{Q}$.  See \eqref{eqn:3.1}. \medskip

For $n = 2$ we have $\mathfrak{D}_{2,3} = \{-35\}$ and
$$h(-35) = 2 = \frac{1}{3}(3^2-3), \ \ \wp_3^2 = \left(\frac{1+\sqrt{-35}}{2}\right).$$

\noindent For $n = 3$ we have $\mathfrak{D}_{3,3} = \{-8, -11, -44, -107\}$ and
$$h(-8)+h(-11)+h(-44)+h(-107) = 1 + 1 + 3 + 3 = 8 =\frac{1}{3}(3^3-3),$$
where
$$\wp_3^3 = \left(\frac{10+\sqrt{-8}}{2}\right), \ \left(\frac{8+2\sqrt{-11}}{2}\right), \ \left(\frac{1+\sqrt{-107}}{2}\right);$$
and the equality occurs in the respective fields $K = \mathbb{Q}(\sqrt{-d})$. \medskip

\noindent For $n= 4$ we have $\mathfrak{D}_{4,3} = \{-56,-224, -260,-323\}$ and
$$h(-56)+h(-224)+h(-260)+h(-323) = 4 + 8 + 8 + 4 = 24 = \frac{1}{3}(3^4-3^2),$$
with
$$\wp_3^4 = \left(\frac{10+2\sqrt{-56}}{2}\right), \ \left(\frac{8+\sqrt{-260}}{2}\right), \ \left(\frac{1+\sqrt{-323}}{2}\right).$$

\noindent For $n=5$ we have that
$$\mathfrak{D}_{5,3} = \{-47, -188,-227, -683\} \cup \{-296, -611, -872\} \cup \{-908, -971\},$$
with
$$\sum_{-d \in \frak{D}_{5,3}}{h(-d)} = 4\cdot 5 + 3 \cdot 10 + 2 \cdot 15 = 80 = \frac{1}{3}(3^5-3),$$
and
\begin{align*}
\wp_3^5 &= \left(\frac{28+2\sqrt{-47}}{2}\right), \ \left(\frac{28+\sqrt{-188}}{2}\right), \ \left(\frac{8+2\sqrt{-227}}{2}\right), \ \left(\frac{17+\sqrt{-683}}{2}\right),\\
&= \left(\frac{26+\sqrt{-296}}{2}\right), \ \left(\frac{19+\sqrt{-611}}{2}\right), \ \left(\frac{10+\sqrt{-872}}{2}\right),\\
&= \left(\frac{8+\sqrt{-908}}{2}\right), \ \left(\frac{1+\sqrt{-971}}{2}\right).
\end{align*}

\noindent For $n = 6$, there are a total of $18$ discriminants in the set $\mathfrak{D}_{6,3}$.  The sum of the class numbers of discriminants in this set is
$$\sum_{-d \in \frak{D}_{6,3}}{h(-d)} = 2 \cdot 2 + 3 \cdot 6 + 8 \cdot 12 +1 \cdot 18 + 4 \cdot 24 = 232 = \frac{1}{3}(3^6-3^3-3^2+3).$$
The discriminants  corresponding to each class number in the set $\{2,6,12,18,24\}$ are:
\begin{align*}
h = 2 \ &: \ -20, -2^2 \cdot 8;\\
h = 6 \ &: \ -4 \cdot 35, -5^2 \cdot 8, -4^2 \cdot 11;\\
h = 12 \ &: \ -7^2 \cdot 20, -10^2 \cdot 8, -7 \cdot 13 \cdot 17, -4 \cdot 13 \cdot 41, -4^2 \cdot 35,\\
& \ \ -5 \cdot 7 \cdot 73, -37 \cdot 71, -8^2 \cdot 11;\\
h = 18 \ &: \ -19 \cdot 89;\\
h = 24 \ &: \ -8^2 \cdot 35, -16^2 \cdot 11, -4 \cdot 23 \cdot 31, -5 \cdot 11 \cdot 53.
\end{align*}

Consider the discriminant $-d = -68$, which does not appear in the above sums.  In the field $K = \mathbb{Q}(\sqrt{-17})$ we have
\begin{align*}
\wp_3^4 & = (8+\sqrt{-17}), \ &2 \cdot 8 \equiv 7 \ (\textrm{mod} \ 9),\\
\wp_3^8 & = (-2143+1504 \sqrt{-17}), \ &2(-2143) \equiv 7 \ (\textrm{mod} \ 9),\\
\wp_3^{12} & = (-509809+36400\sqrt{-17}), \ &2(-509809) \equiv 1 \ (\textrm{mod} \ 9).
\end{align*}
This shows that the discriminant $-d = -68$, with class number $4$, corresponds to the period $n = 12$. \medskip

To give another example, the smallest power of $\wp_3$ that is principal in $\mathbb{Q}(\sqrt{-89})$ (class number $12$) is
$$\wp_3^{12} = (2^7 \cdot 5+37 \sqrt{-89}),$$
but
$$2^8 \cdot 5 \equiv 2 \ (\textrm{mod} \ 9).$$
Hence, 
$$\wp_3^{36} = (28209280 + 40957483 \sqrt{-89}), \ \ 2(28209280) \equiv 8 \ (\textrm{mod} \ 9),$$
shows that the period $n$ corresponding to the discriminant $-4 \cdot 89$ is $n = 36$.
Note that
$$28209280 = 2^7 \cdot 5 \cdot 11 \cdot 4007, \ \ 40957483 = 7^2 \cdot 19 \cdot 29 \cdot 37 \cdot 41.$$
Thus, the discriminants $-4 \cdot 89 \cdot 7^2$ and $-4 \cdot 89 \cdot 7^4$ are also discriminants in $\mathfrak{D}_{36,3}$, with class numbers
$$h(-4 \cdot 89 \cdot 7^2) = 72, \ \ h(-4 \cdot 89 \cdot 7^4) = 7 \cdot 72 = 504,$$
as are the discriminants $-d = -4 \cdot 89 \cdot f^2$, where $f$ is any nonempty product of distinct prime powers in the set $\{7, 7^2, 19, 29, 37, 41\}$.  The largest of these discriminants has class number
$$h(-4 \cdot 89 \cdot 7^4 \cdot 19^2 \cdot 29^2 \cdot 37^2 \cdot 41^2) =  434367360 = 2^7 \cdot 3^6 \cdot 5 \cdot 7^2 \cdot 19.$$

Another fundamental discriminant corresponding to period $n = 36$ is
$$-4d = -4 \cdot 140054575256989121 = -4 \cdot 17 \cdot 23 \cdot 569 \cdot 1291 \cdot 2797 \cdot 174337,$$
since
$$3^{36} = 2^4 \cdot 5^4 \cdot 7^4 \cdot 11^4 \cdot 13^4 + d, \ \ 2^3 \cdot 5^2 \cdot 7^2 \cdot 11^2 \cdot 13^2 \equiv 8 \ (\textrm{mod} \ 9).$$
By genus theory and the fact that $36 \mid 3h(-4d)$ we have that $h(-4d) \ge 3 \cdot 2^6 = 192$.  But no smaller power of $\wp_3$ is principal in $K = \mathbb{Q}(\sqrt{-d})$, so that $36 \mid h(-4d)$ and $h(-4d) \ge 576$.  In fact, a calculation on Magma shows that the class number is
$$h(-4 \cdot 140054575256989121) = 2^8 \cdot 3^2 \cdot 5 \cdot 73 \cdot 479 = 402819840.$$
This value satisfies
$$\frac{h(-4d)}{2d^{1/2}} = .53818576....$$

\section{Testing Conjecture \ref{conj:1}.}

The following theorem shows why it can be tricky to determine when Conjecture \ref{conj:1} holds for a given prime $p$ and a given integer $n \ge 2$.

\begin{thm} Let $p$ be an odd prime which splits in the field $K= \mathbb{Q}(\sqrt{-d})$, where $-d = d_K f^2$ and $(p,d) = 1$.
\begin{enumerate}[(a)]
\item Assume $f>1$ if $d_K = -3$ or $-4$.  The degree of $L_{\mathcal{O},p} = \Sigma_p \Omega_{pf}$ over $K = \mathbb{Q}(\sqrt{-d})$ is given by
\begin{equation}
[L_{\mathcal{O},p}:K] = h(-d) \frac{(p-1)^2}{2}.
\label{eqn:8.1}
\end{equation}
\item If $d_K \neq -3, -4$ and $(p) = \mathfrak{p} \mathfrak{p}'$ in $R_K$, the degree of $\Sigma_\frak{p}\Omega_f$ over $K$ is
\begin{equation}
[\Sigma_\frak{p}\Omega_f:K] = h(-d) \frac{p-1}{2}.
\label{eqn:8.2}
\end{equation}
and the field $F_{\frak{p}'} = \Sigma_{\frak{p}'}\Omega_f$ is the inertia field of $\mathfrak{p}$ in $L_{\mathcal{O},p}/K$.
\item If $d_K = -3$ or $-4$, $w$ is the number of roots of unity in $K$, and $f > 1$, then the field $F_{\frak{p}'} = \Sigma_{\frak{p}'}\Omega_f$ is contained in the inertia field $\textsf{K}_T$ of $\mathfrak{p}$ in $L_{\mathcal{O},p}/K$, and
$$[\textsf{K}_T: F_{\frak{p}'}] = \frac{w}{2}.$$
If $f=1$, $\textsf{K}_T = F_{\frak{p}'}$.
\item If $d_K = -4$ and $2 \mid f$, then the inertia field of $\mathfrak{p}$ in $L_{\mathcal{O},p}/K$ is $\textsf{K}_T = \Sigma_{(2)\frak{p}'}\Omega_f$.
\item If $d_K = -3$ and $3 \mid f$, then the inertia field of $\mathfrak{p}$ in $L_{\mathcal{O},p}/K$ is $\textsf{K}_T = \Sigma'\Omega_f$, where $\Sigma' = \Sigma_{(3)\frak{p}'} \cap \Sigma_p \Omega_{3p}$ is the unique cubic extension of $\Sigma_{\frak{p}'}$ contained in $\Sigma_{(3)\frak{p}'}$.
\end{enumerate}
\label{thm:5}
\end{thm}

\begin{proof}
(a) If $w$ is the number of roots of unity in $K$, then $(p, w) = 1$ and 
$$[\Sigma_p:\Sigma] = \frac{\varphi(\frak{p} \frak{p}')}{w} = \frac{(p-1)^2}{w},$$
since the residue classes mod $p$ represented by roots of unity are distinct.  Thus,
$$[\Sigma_p \Omega_f: \Omega_f] = \frac{\varphi(\frak{p} \frak{p}')}{w} = \frac{(p-1)^2}{w},$$
since $\Omega_f \cap \Sigma_p = \Sigma$.  Furthermore, if $f > 1$,
$$[\Omega_{pf}:\Omega_f] = (p-1);$$
while if $f = 1$,
$$[\Omega_{p}:\Omega_1] = [\Omega_{p}:\Sigma] = \frac{(p-1)}{w/2}.$$
It follows that
$$[\Omega_p \Omega_f:\Omega_f] = \frac{(p-1)}{w/2},$$
since $\Omega_f \cap \Omega_p = \Sigma$, using that $(p,f) = 1$.  Now $\Omega_p \Omega_f \subseteq \Sigma_p \Omega_f \cap \Omega_{pf}$, and we claim that
$$\Omega_p \Omega_f = \Sigma_p \Omega_f \cap \Omega_{pf}.$$
This is clear if $d_K \neq -3, -4$, or $f = 1$, since $\Omega_{pf} = \Omega_p \Omega_f$ in these cases (see \cite{h}).  If $d_K = -3$ or $-4$ and $f > 1$, then
$$[\Omega_{pf}: \Omega_p \Omega_f] = w/2 = 3 \ \textrm{or} \ 2.$$
It follows that if $\Omega_p \Omega_f $ were properly contained in $\Sigma_p \Omega_f \cap \Omega_{pf}$, then we would have that $\Omega_{pf} \subset \Sigma_p \Omega_f$.  But comparing the ideal groups in $K$ for these two extensions, we see that
$$\{(\alpha) \ | \ \alpha \equiv r \ (\textrm{mod} \ pf), r \in \mathbb{Z}\} \nsupseteq \{(\alpha) \ | \ \alpha \equiv r \ (\textrm{mod} \ f), r \in \mathbb{Z}, \alpha \equiv \zeta \ (\textrm{mod} \ p)\};$$
here $\zeta \in K$ is some root of unity and $(r,pf) = 1$.  This proves the claim.  Hence, using that $f>1$ if $d_K = -3,-4$, we have that
\begin{align*}
[\Sigma_p \Omega_{pf}:\Omega_f] &= [\Sigma_p \Omega_{pf}: \Omega_p \Omega_f][\Omega_p \Omega_f: \Omega_f]\\
&= [\Sigma_p \Omega_f: \Omega_p \Omega_f][\Omega_{pf}:\Omega_p \Omega_f][\Omega_p \Omega_f: \Omega_f]\\
&= [\Sigma_p \Omega_f: \Omega_f] [\Omega_{pf}:\Omega_p \Omega_f]\\
&= \frac{(p-1)^2}{w} \frac{w}{2} = \frac{(p-1)^2}{2}.
\end{align*}
This proves (a). \medskip

\noindent (b)  If $d_K \neq -3, -4$, then $w=2$ and
$$[\Sigma_\frak{p} \Omega_f:K] = \frac{p-1}{2} h(-d),$$
since $\Sigma_\frak{p} \cap \Omega_f = \Sigma$.  Now the ramficiation index of $p$ in $\Sigma_p/K$ is at least $e = p-1$, since $\mathbb{Q}(\zeta_p) \subseteq \Sigma_p$.  Hence, the ramification index of $\mathfrak{p}$ in $L_{\mathcal{O},p}/F_{\frak{p}'}$ is at least equal to $[L_{\mathcal{O},p}:F_{\frak{p}'}] = p-1$, since $\mathfrak{p}$ is unramified in $F_{\frak{p}'}/K$.  Thus, it must be totally ramified in $L_{\mathcal{O},p}/F_{\frak{p}'}$.  This proves that $F_{\frak{p}'}$ is the inertia field of $\mathfrak{p}$ in $L_{\mathcal{O},p}/K$. \medskip

\noindent (c)  If $f = 1$, the claim follows as in part (b).  Assume $f > 1$.  Certainly $F_{\frak{p}'} \subseteq \textsf{K}_T$.  The claim will follow if we show that the ramification index of $\mathfrak{p}$ in $L_{\mathcal{O},p}/K$ is equal to $e = p-1$, since
$$[F_{\frak{p}'} :K] = \frac{p-1}{w} h(-d)$$
and therefore
$$[\textsf{K}_T: F_{\frak{p}'}] = \frac{[L_{\mathcal{O},p}:F_{\frak{p}'}]}{[L_{\mathcal{O},p}:\textsf{K}_T]} = \frac{(w/2)(p-1)}{p-1} = \frac{w}{2}.$$
To show $e = p-1$ we first note that $[\Sigma(\zeta_p):\Sigma] = p-1$, so that
$$[\Sigma_p:\Sigma] = \frac{(p-1)^2}{w} \ \textrm{and} \ [\Sigma_{\frak{p}'}:\Sigma] = \frac{(p-1)}{w}$$
imply that
$$[\Sigma_{\frak{p}'} \Sigma(\zeta_p):\Sigma] = \frac{(p-1)^2}{w} = [\Sigma_p:\Sigma],$$
since $\Sigma_{\frak{p}'} \cap \Sigma(\zeta_p) = \Sigma$.  Thus, $\Sigma_p = \Sigma_{\frak{p}'}(\zeta_p)$, which shows that the ramification index of the prime divisors of $\mathfrak{p}$ in $\Sigma_p/\Sigma$ is $e_1 = p-1$.  \medskip

Secondly,
$$[\Omega_{pf}: \Omega_f] = \begin{cases} p-1, \ & f>1;\\ \frac{(p-1)}{w/2}, \ & f = 1. \end{cases}$$
Hence, the ramification index $e_2$ of the prime divisors of $\mathfrak{p}$ in $\Omega_{pf}/\Sigma$ divides $(p-1)$, since $p$ is unramified in $\Omega_f/\Sigma$.  The ramification is tame for both $\Omega_{pf}/\Sigma$ and $\Sigma_p/\Sigma$, so Abhyankar's Lemma \cite[p. 412]{mg}, together with the fact that $L_{\mathcal{O},p} = \Sigma_p \Omega_{pf}$, implies that $e = \textrm{lcm}(e_1, e_2) = p-1$ is the ramification index of the prime divisors of $\mathfrak{p}$ in $L_{\mathcal{O},p}/\Sigma$, and therefore of $\mathfrak{p}$ in $L_{\mathcal{O},p}/K$.  This proves (c). \medskip

\noindent (d) Assuming $f = 2f'$, it follows that $L_{\mathcal{O},p} = \Sigma_p \Omega_{2pf'} = \Sigma_p \Omega_{2p} \Omega_f$ (Hasse \cite{h}).  Checking ideal groups shows that $\Sigma_{(2)\frak{p}'} \subset \Sigma_p \Omega_{2p}$; and $\mathfrak{p}$ is unramified in $\Sigma_{(2)\frak{p}'} \Omega_f \subset L_{\mathcal{O},p}$.  The claim now follows from (c), since $[\Sigma_{(2)\frak{p}'} \Omega_f: \Omega_f] = \frac{p-1}{2} = \frac{w}{2} [\Sigma_{\frak{p}'} \Omega_f: \Omega_f]$.  (Note that $\Sigma_{(2)\frak{p}'} \cap \Omega_f \subseteq \Sigma_{(2)\frak{p}'} \cap \Sigma_{(2f')} \subseteq \Sigma_{(2)} = \Sigma$, since $\varphi(2)/2 = 1$.)\medskip

\noindent (e) As in part (d), we have that $f = 3f'$ and $L_{\mathcal{O},p} = \Sigma_p \Omega_{3pf'} = \Sigma_p \Omega_{3p} \Omega_f$.  Now let $\Sigma'$ be the unique cubic extension of $\Sigma_{\frak{p}'}$ inside $\Sigma_{(3)\frak{p}'}$, which is a $6$-th degree extension of $\Sigma_{\frak{p}'}$.  Then
$$\Sigma_{\frak{p}'} \subset \Sigma' \subset \Sigma_{(3)\frak{p}'}$$
and the corresponding ideal groups (declared modulo $3p$) are
$$H_{(3)\frak{p}'} \subset H' \subset H_{\frak{p}'}.$$
Moreover, $[H':H_{(3)\frak{p}'}] = 2$, so that $H'$ consists of those principal ideals $(\alpha)$, for which (wlog)
$$\alpha \equiv 1 \ (\textrm{mod} \ \mathfrak{p}') \ \ \textrm{and} \ \ (\alpha)^2 \in H_{(3)\frak{p}'} .$$
These conditions imply that $\alpha^2 \equiv \zeta'$ (mod $(3)\mathfrak{p}')$, for some root of unity $\zeta'$ satisfying $\zeta' \equiv 1$ (mod $\mathfrak{p}'$), and therefore $\zeta' = 1$.  Hence
$$H' = \{(\alpha) \ | \ \alpha \equiv 1 \ (\textrm{mod} \ \mathfrak{p}') \ \wedge \ \alpha^2 \equiv 1 \ (\textrm{mod} \ (3)\mathfrak{p}')\}.$$
On the other hand, the ideal group in $K$ of $\Sigma_p \Omega_{3p}$ is
$$H_{3p} = \{(\alpha) \ | \ \alpha \equiv r \ (\textrm{mod} \ 3p) \ \wedge \ \alpha \equiv \zeta \ (\textrm{mod} \ p)\},$$
for some root of unity $\zeta$.  Since $\{1,\zeta_3\}$ is a basis for $R_K = \mathbb{Z}[\zeta_3]$, these conditions imply that $\zeta = \pm 1$.  
If necessary, we replace $\alpha$ by $-\alpha$ and obtain that
$$H_{3p} = \{(\alpha) \ | \ \alpha \equiv r \ (\textrm{mod} \ 3p) \ \wedge \ \alpha \equiv 1 \ (\textrm{mod} \ p)\}.$$
Now it is clear that, for these $\alpha$, we have $\alpha \equiv \pm 1$ (mod $3$), hence $\alpha^2 \equiv 1$ (mod $3p$).  This implies that $H_{3p} \subset H'$ and therefore $\Sigma' \subset \Sigma_p \Omega_{3p}$.  This gives that
$$\Sigma' \subseteq \Sigma_{(3)\frak{p}'} \cap \Sigma_p \Omega_{3p}$$
and therefore
$$\Sigma'\Omega_f \subseteq \Sigma_{(3)\frak{p}'}\Omega_f \cap \Sigma_p \Omega_{3p}\Omega_f.$$
It is clear that the latter intersection is contained in $\textsf{K}_T$, which has degree $3$ over $ \Sigma_{\frak{p}'}\Omega_f$, by part (c).  This implies that the intersection equals $\textsf{K}_T = \Sigma'\Omega_f$, since the latter field also has degree $3$ over $\Sigma_{\frak{p}'}\Omega_f$.  
(For this, note that $\Omega_f \cap \Sigma_{(3)\mathfrak{p}'}$ is an abelian extension of $K$ whose conductor divides $(f,(3)\mathfrak{p}') = (3)$, and is therefore a subfield of $\Sigma_{(3)} = \Sigma$.  Hence, $\Omega_f \cap \Sigma_{(3)\mathfrak{p}'} = \Sigma$.  Now appeal to \cite[Satz 119, pp. 122-123]{halg} or \cite[Prop. 19, p. 591]{df}.) This proves (e).
\end{proof}

\noindent {\bf Example ($p = 5$).}  Consider the case $p = 5$.  In \cite{m6} it was proved that for any discriminant $-d = d_K f^2 \equiv \pm 1$ (mod $5$), the inertia field of a prime divisor $\wp_5$ of $(5)$ in $L_{\mathcal{O},5}/K$, $K = \mathbb{Q}(\sqrt{-d})$, is generated over $\mathbb{Q}$ by a value $\eta = r(w/5)$ of the Rogers-Ramanujan continued fraction $r(\tau)$.  Thus, when $-d = -4f^2$, with odd $f$, $\textsf{K}_T = \mathbb{Q}(\eta)$, where $\eta = r((a+fi)/5)$, where $5^2 \mid a^2+f^2, (a,f) = 1$, and $\wp_5^2 \mid (a+fi)$.  When $f = 3$, $\textsf{K}_T$ is generated over $\mathbb{Q}$ by the root $x = r((-4+3i)/5)$ of the polynomial
\begin{equation*}
p_{36}(x) = x^8 +x^6 - 6x^5 +9x^4 +6x^3 +x^2 +1.
\end{equation*}
See \cite[pp. 1208, 1210]{m6}.  In this case, the congruence
\begin{equation*}
p_{36}(x) \equiv (x^4 + 3x^3 + x^2 + 2x + 1)(x + 3)^4 \ (\textrm{mod} \ 5)
\end{equation*}
shows that the inertial degree of $\wp_5$ in $\textsf{K}_T/K$ is $4$, while the inertial degree of $\wp_5$ in $\Sigma_{\wp_5'}\Omega_f = \Omega_3$ is $2$, since
$$\wp_5 = (2+i), \ \wp_5^2 = (-4+3i).$$
On the other hand, from \cite[p. 810]{m2}, if $f = 7$, we have
\begin{align*}
p_{196}(x) &= x^{16} + 14x^{15} + 64x^{14} + 84x^{13} - 35x^{12} - 14x^{11} + 196x^{10} + 672x^9\\
& \ \ \  + 1029x^8 - 672x^7 + 196x^6 + 14x^5 - 35x^4 - 84x^3 +64x^2 - 14x+1.
\end{align*}
Here, the inertial degree of $\wp_5$ in $\textsf{K}_T/K$ is $8$, since
$$p_{196}(x) \equiv (x + 3)^8(x^8 + 2x^6 + 4x^5 + x^3 + 2x^2 + 1) \ (\textrm{mod} \ 5);$$
while the inertial degree of $\wp_5$ in $\Omega_7$ is $4$, because
$$\wp_5^4 = (-i(2+i)^4) = 24+7i.$$
In both of these cases, the ideal $\wp_5$ remains prime in $\textsf{K}_T$.  It follows from these calculations that the discriminant $-d = -36$ contributes $h(-36) = 2$ to the sum \eqref{eqn:1.3} over discriminants in $\mathfrak{D}_{4,5}$, while $-d = -196$ contributes $h(-196) = 4$ to the corresponding sum over $\mathfrak{D}_{8,5}$.  \medskip

However, if $p = 5$ and $f = 11$, $\textsf{K}_T$ is generated over $\mathbb{Q}$ by a root of the polynomial
\begin{align*}
p_{484}(x) &= x^{24} + 22x^{23} + 89x^{22} - 1452x^{21} + 946x^{20} + 10890x^{19} - 2706x^{18}\\
& \  - 18106x^{17} - 12309x^{16} + 20570x^{15} + 119702x^{14} - 36322x^{13} - 207713x^{12} \\
& \  + 36322x^{11} + 119702x^{10} - 20570x^9 - 12309x^8 + 18106x^7 - 2706x^6 \\
& \ - 10890x^5 + 946x^4 + 1452x^3 + 89x^2 - 22x + 1.
\end{align*}
This may be verified by noting that $\eta = r((2+11i)/5)$ is a periodic point with period $n = 3$ of the algebraic function $\mathfrak{g}(z)$ considered in $\cite{m2}$.  The polynomial $p_{484}(x)$ is one of four polynomials of degree $4h(-d) = 24$ dividing the corresponding polynomial $\textsf{P}_3(x)$, and the only one of the four whose discriminant is divisible by $11$.  We have the following congruence modulo $5$:
$$p_{484}(x) \equiv (x + 3)^{12}(x^3 + 4x^2 + 3)(x^3 + 3x + 3)(x^3 + 3x^2 + 2x + 2)(x^3 + 4x^2 + 4x + 2),$$
so that $\wp_5$ splits into a product of $4$ prime ideals of degree $3$ in $\textsf{K}_T$.  In addition, the quantity $z = \eta-1/\eta$ generates $\Omega_{11}$ over $\mathbb{Q}$, and is a root of the minimal polynomial
\begin{align*}
q_{484}(z) &= z^{12} + 22z^{11} + 101z^{10} - 1210z^9 + 1890z^8 - 1210z^7 + 8089z^6 + 20614z^5\\
& \ \  - 5070z^4 + 40150z^3 + 63509z^2 - 173z + 3733.
\end{align*}
Since
$$q_{484}(z) \equiv (z + 1)^6(z^3 + 2z^2 + z + 3)(z^3 + 4z^2 + z + 1) \ (\textrm{mod} \ 5),$$
the prime $\wp_5$ splits into $2$ primes of degree $3$ in the $6$-th degree extension $\Omega_{11}/K$.  This can also be seen from
$$\wp_5 = (2+i), \ \wp_5^3 = (2+11 i).$$
In this example, we see that the prime divisors of $5$ in $F_{\wp_5'} = \Omega_f$ split in the extension $\textsf{K}_T/\Omega_f$.  Thus, both types of splitting behavior are possible in this context, for the field $K = \mathbb{Q}(i)$ and odd values of the ring conductor $f$.  In particular, the inertial degree of the prime divisors of $\wp_5$ in $\textsf{K}_T$ is equal to the period of the value $\eta$ with respect to the algebraic function $\mathfrak{g}(z)$.  The same relationship holds for all values of $f$.  In addition, the discriminant $-d = -484 \in \mathfrak{D}_{3,5}$ contributes $h(-484) = 6$ to the sum \eqref{eqn:1.3} for $n = 3$. \medskip

To consider the case $p = 7$, we note the following theorem, which will be proved in another paper.  See \cite[Thm. 5, p. 362]{m7} and \cite[p. 156]{duk}.

\begin{thm}
Let $h(\tau)$ be the modular function for $\Gamma_1(7)$ defined by
$$h(\tau) = q^{-1} \prod_{n \ge 1}{\frac{(1-q^{7n-3})(1-q^{7n-4})(1-q^{7n-2})^2(1-q^{7n-5})^2}{(1-q^{7n-1})^3(1-q^{7n-6})^3}}, \ q = e^{2 \pi i \tau}.$$
If $\left(\frac{-d}{7}\right) = +1$ and $(7) = \frak{p}_7 \frak{p}_7'$ in $K = \mathbb{Q}(\sqrt{-d})$, with $-d = d_Kf^2$, let $\mathcal{O} = \textsf{R}_{-d}$ be the order of discriminant $-d$ in $K$ and 
$$w = \frac{v+\sqrt{-d}}{2}, \ \ \wp_7^2 \mid w, \ \ (w,f) = 1.$$
Then $\mathbb{Q}(h(w/7))$ is the inertia field of the prime ideal $\frak{p}_7$ in the extension $L_{\mathcal{O},7} = \Sigma_7 \Omega_{7f}$ over $K$.  The inertial degree of $\frak{p}_7$ in $\mathbb{Q}(h(w/7))/K$ is equal to the period of $\eta = h(w/7)$ with respect to the algebraic function defined by
\begin{align*}
F(x,y) &= -y^7 + (x^7 - 21x^6 + 161x^5 - 546x^4 + 791x^3 - 406x^2 + 21x + 16)y^6\\
& \ + (-16x^7 + 280x^6 - 1477x^5 + 2380x^4 - 2331x^3 + 1330x^2 - 182x - 74)y^5\\
& \ + (74x^7 - 1085x^6 + 4725x^5 - 5866x^4 + 3171x^3 - 1148x^2 + 203x + 78)y^4\\
& \ + (-78x^7 + 1211x^6 - 5642x^5 + 7378x^4 - 3066x^3 + 56x^2 + 63x - 9)y^3\\
& \  + (9x^7 - 308x^6 + 2247x^5 - 3549x^4 + 1491x^3 + 203x^2 - 84x - 10)y^2\\
& \  + (10x^7 - 84x^6 + 7x^5 + 168x^4 - 21x^3 - 56x^2 - 14x - 1)y + x^7.
\end{align*}
\label{thm:6}
\end{thm}

We will use this theorem to check Conjecture 1 for $p = 7$ and $n = 2,3$. \medskip

First, we have the following generalization of Theorem \ref{thm:4}.

\begin{thm}
Assume $d_K \neq -3, -4$ and $(p) = \mathfrak{p} \mathfrak{p}'$ in $R_K$.  Then the degree $n$ over $K$ of the prime divisors of $\mathfrak{p}$ in $F_{\frak{p}'} \subset L_{\mathcal{O},p}$ is the smallest positive integer $n$ for which
$$\mathfrak{p}^n = \left(\frac{x+y\sqrt{-d}}{2}\right), \ \frac{x+y\sqrt{-d}}{2} \in \mathcal{O} = \textsf{R}_{-d}, \ \textrm{with} \ x \equiv \pm 1 \ (\textrm{mod} \ p),$$
and $\alpha = \frac{x+y\sqrt{-d}}{2} \equiv \pm 1$ (mod $\mathfrak{p}'$).  On the other hand, if $d_K = -3$ or $-4$ and $n$ is the degree of the prime divisors of 
$\frak{p}$ in $\textsf{K}_T \subset  L_{\mathcal{O},p}$, then the above equation still holds, but with
\begin{equation*}
x \equiv \zeta \ (\textrm{mod} \ \frak{p}'), \ \textrm{for some} \ \zeta = \pm \rho^r \ \textrm{or} \ i^r, \  r \in \mathbb{Z},
\end{equation*}
according as $d_K = -3$ or $-4$.  This condition is necessary but not sufficient.
\label{thm:7}
\end{thm}

\begin{proof}
For discriminants $d_K \neq -3, -4$ the proof is the same as the proof of Theorem \ref{thm:4}.  If $d_K = -3$ or $-4$, then as before, $\frak{p}^n$ has to be a principal ideal generated by an element in the order $\textsf{R}_{-d}$; and as before, we have that
$$\alpha = \frac{x+y\sqrt{-d}}{2} \equiv x \ (\textrm{mod} \ \frak{p}'), \ \ \textrm{since} \ \sqrt{-d} \equiv \frac{x}{y} \  (\textrm{mod} \ \frak{p}').$$
But in this case, the fact that the degree of $\frak{p}$ in $\Sigma_{\frak{p}'}/K$ divides $n$ implies that $\frak{p}^n$ lies in the ideal group corresponding to $\Sigma_{\frak{p}'}$, hence that $\alpha \equiv \zeta$ (mod $\frak{p}'$), where $\zeta$ is a root of unity in $K$.  The last statement of the theorem follows from the fact that the inertial degree of the prime divisors of $\frak{p}$ in $\textsf{K}_T$ can be a multiple of the degree of its prime divisors in $F_{\frak{p}'} = \Sigma_{\frak{p}'} \Omega_f$.
\end{proof}

For primes $p \equiv 11$ (mod $12$) this theorem allows us to compute the left side of \eqref{eqn:1.5} fairly easily, since the discriminants $d_K = -3,-4$ do not arise for these primes.  The results for $n = 2, 3, 4$ for various primes are given in Tables \ref{tab:1}-\ref{tab:3}, in which the left (LHS) and right (RHS) hand sides of \eqref{eqn:1.5} have been computed and compared, with the difference $LHS-RHS$ listed in the last column.  \medskip

These computations give evidence that Conjecture 1 holds for $p = 11$, but does not hold for the other primes in the tables.  \medskip

\noindent {\bf Example ($p = 7$).}  We first use Theorem \ref{thm:7} to determine the discriminants other than $-d = -3f^2$ which lie in $\mathfrak{D}_{2,7}$.  We have
\begin{align*}
4 \cdot 7^2 &= 1+ 3 \cdot 5 \cdot 13;\\
&= 6^2 + 2^5 \cdot 5;\\
&= 8^2 + 2^2 \cdot 3 \cdot 11.
\end{align*}
Thus we have the discriminants 
$$\{-3 \cdot 5 \cdot 13, \ -2^3 \cdot 5, \ -2^5 \cdot 5, \ -4 \cdot 3 \cdot 11\} \subseteq \mathfrak{D}_{2,7}.$$
Note that none of these discriminants occur for $n = 1$, since $4 \cdot 7-1 = 27$.  The sum of the class numbers for these four discriminants is
$$h(-15 \cdot 13)+h(-2^3 \cdot 5)+h(-2^5 \cdot 5) + h(-12 \cdot 11) = 4 + 2 + 4 + 4 = 14,$$
which is already equal to the right side of \eqref{eqn:1.5}.  Thus, we need to see if any discriminants of the form $-d = -3f^2$ lie in $\mathfrak{D}_{2,7}$.
\medskip

Using Theorem \ref{thm:6} we compute the resultant $R_2(x) = \textrm{Res}_y(F(x,y),F(y,x))$, which has exactly four primitive irreducible factors (i.e., whose roots are not fixed points), whose degrees are $6h(-d) = 12, 24, 24, 24$ (see \eqref{eqn:8.2}).  Since the above four discriminants must correspond to irreducible factors of $R_2(x)$, these are the only discriminants in the set $\mathfrak{D}_{2,7}$.  This verifies \eqref{eqn:1.5} for $p = 7$ and $n = 2$.  However, it is instructive to try to see this directly.  In $K = \mathbb{Q}(\sqrt{-3})$ we have that $\frak{p}_7 = (2-\rho)$, say, where $\rho = \frac{-1+\sqrt{-3}}{2}$.  The other generators of $\frak{p}_7$ are, up to sign,
\begin{equation}
2-\rho, \ \rho(2-\rho) = 1+3\rho, \ \textrm{and} \ \rho^2(2-\rho) = -3 - 2\rho.
\label{eqn:8.3}
\end{equation}
Since $\rho \equiv -3$ (mod $\frak{p}_7'$), these generators are congruent respectively, to
$$2 - \rho \equiv 5, \ \ 1 + 3\rho \equiv -1, \ \ -3 - 2\rho \equiv 3 \ \ (\textrm{mod} \ \frak{p}_7').$$
The middle congruence shows that $\frak{p}_7$ splits in $F_{\frak{p}'} = \Sigma_{\frak{p}_7'}\Omega_3$.  However, since $f = 3$ we must use the criterion in Theorem \ref{thm:5}(e).  What is the inertial degree of $\frak{p}_7$ in $\Sigma' \Omega_3$?  From the proof of that result, the ideal group of $\Sigma'$ in $K$ is
$$H' = \{(\alpha) \ | \ \alpha \equiv 1 \ (\textrm{mod} \ \frak{p}_7') \ \wedge \ \alpha^2 \equiv 1 \ (\textrm{mod} \ (3)\frak{p}_7')\}.$$
The generator $\alpha = -(1+3\rho)$ does satisfy both congruences, which implies that $\frak{p}_7 = (\alpha) \in H'$.  Thus $\frak{p}_7$ splits in $\Sigma'$ and clearly also in $\Omega_3$, so its inertial degree with respect to $\Sigma' \Omega_3$ is $1$, not $2$.  Therefore, $-d = -3^3 \notin \mathfrak{D}_{2,7}$.  For the same reasons, using Theorem \ref{thm:5}(c), with $f = 1$, $-d = -3$ is also not in this set.  These two discriminants correspond to the equation $4\cdot 7 =1 + 3^3$. \medskip

This leaves open the question of whether any of the discriminants implied by the relations
\begin{equation*}
\frak{p}_7^2 = (2-\rho)^2 = (3-5\rho) = (5+8\rho) = (-8-3\rho),
\end{equation*}
lie in $\mathfrak{D}_{2,7}$.  We can ignore the ring conductors $f = 1,3$, for the reasons given above, but what can we say about the ring conductors $f = 5, 2, 4, 8$?  These relations show that $\frak{p}_7$ splits in $\Sigma_{\frak{p}_7'}$ and its prime divisors have degree $2$ in the field $\Omega_f$, for $f = 5, 4, 8$.  Thus, because $\frak{p}_7$ is unramified in both extensions of $K$, its prime divisors have degree $2$ in the composite field $\Sigma_{\frak{p}_7'} \Omega_f$.  However, in this case the inertia field $\textsf{K}_T$ of $\frak{p}_7$ in $L_{\mathcal{O},7}$ has degree $3$ over $F_{\frak{p}'} = \Sigma_{\frak{p}_7'} \Omega_f$, and we don't have an easily applicable criterion here for determining the inertial degree of $\frak{p}_7$ in $\textsf{K}_T$, unless we have a polynomial whose root generates this field (see the continuation of this example below)\footnote{It turns out that all of these conductors, except $f = 2$, correspond to discriminants in $\mathfrak{D}_{6,7}$, i.e., to the period $n = 6$.}, or unless we can specify its ideal group in $K$.  The following proposition shows this is a nontrivial question.  At any rate, Theorem \ref{thm:6}, and in particular, the existence of the algebraic function defined by $F(x,y) = 0$, shows that there are no other discriminants in the set $\mathfrak{D}_{2,7}$.

\begin{prop}
Let $L_1$ and $L_2$ be distinct finite abelian extensions of the $p$-adic field $\mathbb{Q}_p$, with residue class degrees $f_1, f_2$ and tame ramification $e_1 = e_2 = e$.  Further, let $F_1$ and $F_2$ be the inertia fields of $\mathbb{Q}_p$ contained in $L_1, L_2$, respectively, over which $L_1$ and $L_2$ are purely ramified:
$$[L_1: F_1] = e = [L_2: F_2].$$
Let $L = L_1 L_2$ be the composite extension in $\overline{\mathbb{Q}}_p$ and $F$ its inertia field.  Then $[F:F_1 F_2]$ is a divisor of $e$.  If, in addition, $L_1 \cap L_2 = \mathbb{Q}_p$, then $[F:F_1F_2] = e$.
\label{prop:4}
\end{prop}

\begin{proof}
By \cite[p. 242]{hz1} (or see \cite[p. 251]{hz2}, in which the counting arguments in \cite[Ch. 16]{hz1} have been corrected), the fields $L_1, L_2$ have the form
$$L_1 = \mathbb{Q}_p(\omega_1,\sqrt[e]{p \rho_1}), \ \ L_2 = \mathbb{Q}_p(\omega_2,\sqrt[e]{p \rho_2}),$$
where $\omega_i$ is a primitive $n_i$-th ($n_i = p^{f_i}-1$) root of unity contained in $L_i$, and $\rho_i = \omega_i^{r_i}$; and where $e \mid (p-1)$, since $L_i$ is abelian over $\mathbb{Q}_p$.  Further, $F_i = \mathbb{Q}_p(\omega_i)$.  If $\pi_i = \sqrt[e]{p \rho_i}$, then $(\pi_1/\pi_2)^e = \rho_1/\rho_2 = \zeta$ is an $n$-th ($n = p^f-1$) root of unity, where $f = \textrm{lcm}(f_1, f_2)$.  In particular, $\zeta^{1/e} \in L$, and
$$L =\mathbb{Q}_p(\omega_1, \omega_2, \zeta^{1/e}, \sqrt[e]{p \rho_2}),$$
since $\pi_2 \zeta^{1/e} = \pi_1$.  It follows that $F = \mathbb{Q}_p(\omega_1,\omega_2, \zeta^{1/e}) = F_1F_2(\zeta^{1/e})$ is unramified over $\mathbb{Q}_p$, 
since $e \mid (p-1)$, and therefore $[F:F_1F_2] \mid e$.  Since $L$ is purely ramified over $F$, this proves the first assertion.  \medskip

For the second assertion, $L/L_1$ is unramified, since the ramification indices of $L$ and $L_1$ over $\mathbb{Q}_p$ are both equal to $e$.  Now $[L:L_1] = [L_2:\mathbb{Q}_p] = ef_2$ implies that the inertial degree of $L/\mathbb{Q}_p$ is $f_1f_2 e$, where $[F_1F_2:\mathbb{Q}_p] = f_1f_2$.  This proves the second assertion.
\end{proof}

This proposition gives examples where the residue class field of a composition is larger than the composition of the individual residue class fields. In our case, the fields $\Sigma_{\frak{p}'}$ and $\Omega_f$ are the inertia fields for $\frak{p}$ in $\Sigma_p$ and $\Omega_{pf}$, respectively (see the proof of Theorem \ref{thm:5} and the proposition below), but the inertia field for the composite extension $\Sigma_p \Omega_{pf}$ is larger than $F_{\frak{p}'} = \Sigma_{\frak{p}'} \Omega_f$.  In this case $e = p-1$ is the ramification index for both both individual extensions, and $[\textsf{K}_T: F_{\frak{p}'}] = w/2$ divides $p-1$, when $K = \mathbb{Q}(\sqrt{-3})$ or $\mathbb{Q}(\sqrt{-1})$, as shown in Proposition \ref{prop:4}. For $p > 3$ and $f > 1$, the final assertion of the proposition, together with $w/2 < p-1$, implies that the completions $L_1 = \Sigma_{p,\frak{q}_1}$ and $L_2 = \Omega_{pf,\frak{q}_2}$, for prime divisors $\frak{q}_i$ of $\frak{p}$ in $\Sigma_p$, respectively, $\Omega_{pf}$, have an intersection strictly larger than $\mathbb{Q}_p$. \medskip

To show that $\Omega_{pf}/\Omega_f$ is totally ramified, we prove the following.

\begin{prop} Let $H_f$, for some $f > 1$, denote the ideal group for the ring class field $\Omega_f$ over $K = \mathbb{Q}(\zeta)$, $\zeta = i$ or $\rho = \frac{-1+\sqrt{-3}}{2}$, and let the ideal group $H_{pf}$ correspond to $\Omega_{pf}$, where $p$ is an odd prime not dividing $3f$ which splits in $K$.  Also assume that $2 \nmid f$ if $d_K = -4$ and $3 \nmid f$ if $d_K = -3$.
\begin{enumerate}[(a)]
\item The quotient group $H_f/H_{pf}$ is cyclic.
\item $H_f$ is the smallest ideal group containing $H_{pf}$ whose conductor is relatively prime to $\mathfrak{p}$, where $(p) = \mathfrak{p} \mathfrak{p}'$.
\item The inertial degree of $\mathfrak{p}$ in $\Omega_{pf}/K$ equals the inertial degree of $\mathfrak{p}$ in $\Omega_f/K$, and the prime divisors of $\mathfrak{p}$ in $\Omega_f$ are totally ramified in $\Omega_{pf}/\Omega_f$.
\end{enumerate}
\label{prop:5}
\end{prop}

\begin{proof}
(a) We begin with the isomorphism
$$(R_K/(pf))^* \cong (R_K/\mathfrak{p})^* \times (R_K/\mathfrak{p}')^* \times (R_K/(f))^*.$$
Let $D = \{(r,r) \in (R_K/\mathfrak{p})^* \times (R_K/\mathfrak{p}')^*, \ r \in \mathbb{Z}, (r,pf) = 1\}$.  We show that $(R_K/(p))^*/D$ is cyclic.  For this let $g \in \mathbb{Z}$ be a primitive root modulo $p$ with $(g,pf) = 1$, and consider the order of $xD = (g,1)D$ in $((R_K/\mathfrak{p})^* \times (R_K/\mathfrak{p}')^*)/D$.  We have that $x^{p-1}D = (1,1)D = D$ and $x^kD = (g^k, 1)D \neq D$, for $0 < k< p-1$.  This shows that $((R_K/\mathfrak{p})^* \times (R_K/\mathfrak{p}')^*)/D$ is cyclic of order $p-1$. \medskip

Now we view
\begin{align*}
\tilde{H}_f &= \{[\alpha] \in (R_K/(pf))^* \ | \ (\alpha,p) = 1 \ \wedge \ \alpha \equiv r \ \textrm{mod} \ f, \ r \in \mathbb{Z} \}\\
 &\cong (R_K/\mathfrak{p})^* \times (R_K/\mathfrak{p}')^* \times \{[\alpha] \in (R_K/(f))^* \ | \ \alpha \equiv r \ \textrm{mod} \ f, r \in \mathbb{Z}\}
 \end{align*}
and
\begin{align*}
\tilde{H}_{pf} &= \{[\alpha] \in (R_K/(pf))^* \ | \ \alpha \equiv r \ \textrm{mod} \ pf, \ r \in \mathbb{Z} \}\\
&\cong D \times \{[\alpha] \in (R_K/(f))^* \ | \ \alpha \equiv r \ \textrm{mod} \ f, r \in \mathbb{Z}\}\\
&= D \times D_f
\end{align*}
as subgroups of $(R_K/(pf))^*$.  Then
$$\tilde{H}_f/\tilde{H}_{pf} \cong ((R_K/\mathfrak{p})^* \times (R_K/\mathfrak{p}')^*)/D \times D_f/D_f,$$
which implies that $\tilde{H}_f/\tilde{H}_{pf}$ is cyclic.  Now let $\alpha = x+fy\zeta$ ($\zeta = i$ or $\rho$) be chosen so that $[\alpha] \in \tilde{H}_f$ generates $\tilde{H}_f/\tilde{H}_{pf}$.  Then $(\alpha) \in H_f$ and $(\alpha)^{p-1} \in H_{pf}$.  Furthermore, no smaller power $(\alpha)^k$, $0, < k < p-1$, lies in $H_{pf}$.  If it did, then $\alpha^k \equiv r\zeta'$ (mod $pf$), for some root of unity $\zeta'$ and $r \in \mathbb{Z}$.  But $\alpha \equiv x$ (mod $f$), where $x \in \mathbb{Z}$, which implies that $x^k \equiv r\zeta'$ (mod $f$), so that $\zeta' = \pm 1$.  In that case $\alpha^k \equiv \pm r$ (mod $pf$), so that $(\alpha)\tilde{H}_{pf}$ would have order smaller than $p-1$ in $\tilde{H}_f/\tilde{H}_{pf}$.  This proves (a). \medskip

(b) If $H$ is an ideal group, declared modulo $pf$, satisfying $H_{pf} \subset H \subsetneq H_f$, then $H^\sigma = H$ for the nontrivial automorphism $\sigma$ of $K/\mathbb{Q}$.  This is because $H_f, H_{pf}$ are left fixed by $\sigma$ and $H$ is the unique subgroup of $H_f$ of index $[H_f:H]$, by a).  It follows that the conductor of $H$ would be a rational integer divisible by $f$ and dividing $pf$.  It is easy to see that this conductor cannot be $f$; if it were, $H$ would consist of principal ideals $(\alpha)$, where $\alpha \equiv r$ (mod $f$) for a proper subset of reduced residues $r$ modulo $f$, while any reduced residue (mod $f$) (and prime to $p$) can occur for an ideal $(\beta) \in H_{pf}$.  Hence, the conductor of such an $H$ must be divisible by $p$.  It follows that $H_f$ is the smallest ideal group (declared modulo $pf$) containing $H_{pf}$ whose conductor is not divisible by $\mathfrak{p}$.  This proves (b). \medskip

(c) The first assertion follows from part (b) and Hasse's theorem \cite[Satz 12, p. 31]{h0}, \cite[p. 137]{h2}, which says that the smallest power of $\mathfrak{p}$ lying in $H_f$ is the inertial degree of $\mathfrak{p}$ in $\Omega_{pf}/K$.  The second also follows from Hasse's theorem, since the ramification index of prime 
divisors of $p$ in $\Omega_{pf}/\Omega_f$ is $e = [H_f: H_{pf}] = p-1 = [\Omega_{pf}: \Omega_f]$, where the ideal groups $H_f, H_{pf}$ in $K$ correspond to the class 
fields $\Omega_f, \Omega_{pf}$ over $K$.
\end{proof}

\begin{table}[htp]
\caption{Checking Conjecture 1 for various values of $p$ with $n=2$.}
\begin{center}
\begin{tabular}{|c|c|c|c|c|}
\hline
$p$&Discriminants&LHS&RHS&Difference\\
\hline
$11$&$h(-2^3\cdot3)=2,\ h(-2^5\cdot3)=4,\ h(-2^7\cdot3)=8$,&$22$&$22$&$0$\\
&$h(-3\cdot7\cdot23)=4,\ h(-4\cdot5\cdot17)=4$&&&\\\hline

$23$&$h(-5\cdot47)=2,\ h(-3^2\cdot5\cdot47)=8$,&$30$&$46$&$-16$\\
&$h(-2^3\cdot3\cdot17)=4,\ h(-2^5\cdot3\cdot17)=8$,&&&\\
&$h(-4\cdot5\cdot7\cdot11)=8$&&&\\\hline

$47$&$h(-3\cdot5\cdot19\cdot31)=16,\ h(-4\cdot3\cdot5\cdot7)=8$,&$88$&$94$&$-6$\\
&$h(-2^4\cdot3\cdot5\cdot7)=16,\ h(-2^6\cdot3\cdot5\cdot7)=32$,&&&\\
&$h(-4\cdot23\cdot71)=16$&&&\\\hline

$59$&$h(-7\cdot13\cdot17)=12,\ h(-3^2\cdot7\cdot13\cdot17)=24$,&$128$&$118$&$10$\\
&$h(-4\cdot3\cdot5\cdot11)=8,\ h(-2^4\cdot3\cdot5\cdot11)=16$,&&&\\
&$h(-2^6\cdot3\cdot5\cdot11)=32,\ h(-4\cdot29\cdot89)=36$&&&\\\hline

$71$&$h(-3\cdot11\cdot13\cdot47)=32,\ h(-4\cdot5\cdot7\cdot107)=24$,&$146$&$142$&$4$\\
&$h(-2^3\cdot53)=6,\ h(-2^3\cdot3^2\cdot53)=24$,&&&\\
&$h(-2^5\cdot53)=12,\ h(-2^5\cdot3^2\cdot53)=48$&&&\\\hline

$83$&$h(-3\cdot5\cdot11\cdot167)=40,\ h(-4\cdot5\cdot41)=8$,&$136$&$166$&$-30$\\
&$h(-4\cdot5^3\cdot41)=40,\ h(-2^3\cdot3\cdot7\cdot31)=16$,&&&\\
&$h(-2^5\cdot3\cdot7\cdot31)=32$&&&\\\hline
\end{tabular}
\end{center}
\label{tab:1}
\end{table}\bigskip\pagebreak

\begin{table}[htp]
\caption{Checking primes $p = 11, 23$ for $n=3$.}
\begin{center}
\begin{tabular}{|c|c|c|c|c|}
\hline
$p$&Discriminants&LHS&RHS&Difference\\
\hline
$11$&$h(-5323)=15,\ h(-2^3\cdot653)=18$,&$264$&$264$&$0$\\
&$h(-5\cdot7\cdot37)=36,\ h(-4\cdot5\cdot7\cdot37)=36$,&&&\\
&$h(-19\cdot257)=18,\ h(-5\cdot7\cdot137)=12$,&&&\\
&$h(-4\cdot43)=3,\ h(-5^2\cdot43)=6$,&&&\\
&$h(-4\cdot547)=9,\ h(-4\cdot5^2\cdot43)=18$,&&&\\
&$h(-2^3\cdot521)=12,\ h(-139)=3$,&&&\\
&$h(-5^2\cdot139)=12,\ h(-3299)=27$,&&&\\
&$h(-2^3\cdot743)=24,\ h(-547)=3$,&&&\\
&$h(-7\cdot157)=6,\ h(-5\cdot167)=6$&&&\\\hline

$23$&$h(-41\cdot1187)=36,\ h(-2^3\cdot19\cdot317)=60$,&$1149$&$1104$&$45$\\
&$h(-11\cdot1093)=84,\ h(-4\cdot11\cdot1093)=84$,&&&\\
&$h(-46643)=45,\ h(-7\cdot6637)=60$,&&&\\
&$h(-4\cdot7\cdot13)=6,\ h(-7\cdot11^2\cdot13)=24$,&&&\\
&$h(-4\cdot7\cdot11^2\cdot13)=72,\ h(-2^3\cdot5471)=48$,&&&\\
&$h(-40387)=27,\ h(-7\cdot5717)=78$,&&&\\
&$h(-2^3\cdot7\cdot13)=12,\ h(-2^3\cdot7^3\cdot13)=84$,&&&\\
&$h(-8803)=9,\ h(-4\cdot8803)=27$,&&&\\
&$h(-29\cdot1031)=36,\ h(-29347)=27$,&&&\\
&$h(-73\cdot79)=36,\ (-4\cdot73\cdot79)=36$,&&&\\
&$h(-2^3\cdot2803)=126,\ h(-43\cdot353)=60$,&&&\\
&$h(-11\cdot13\cdot101)=12,\ h(-2^3\cdot19\cdot41)=12$,&&&\\
&$h(-7\cdot193)=24,\ (-4\cdot7\cdot193)=24$&&&\\\hline
\end{tabular}
\end{center}
\label{tab:2}
\end{table}\bigskip\pagebreak

\begin{table}[htp]
\caption{Checking $p = 11$ for $n=4$.}
\begin{center}
\begin{tabular}{|c|c|c|c|c|}
\hline
$p$&Discriminants&LHS&RHS\\
\hline
$11$&$h(-3\cdot241)=4,\ h(-3^3\cdot241)=12$,&$2904$&$2904$\\
&$h(-3^5\cdot241)=36,\ h(-2^3\cdot7\cdot29)=16$,&&\\
&$h(-2^5\cdot7\cdot29)=32,\ h(-2^3\cdot3^2\cdot7\cdot29)=64$,&&\\
&$h(-2^5\cdot3^2\cdot7\cdot29)=128,\ h(-4\cdot5\cdot23\cdot127)=80$,&&\\
&$h(-13\cdot17\cdot263)=32,\ h(-3\cdot5\cdot53\cdot73)=64$,&&\\
&$h(-4\cdot3\cdot5\cdot7\cdot137)=64,\ h(-4\cdot3\cdot13\cdot23)=16$,&&\\
&$h(-2^4\cdot3\cdot13\cdot23)=32,\ h(-2^6\cdot3\cdot13\cdot23)=64$,&&\\
&$h(-3\cdot5\cdot19\cdot199)=48,\ h(-7\cdot41\cdot197)=56$,&&\\
&$h(-2^3\cdot37\cdot47)=24,\ h(-2^5\cdot37\cdot47)=48$,&&\\
&$h(-4\cdot3\cdot31\cdot149)=64,\ h(-3\cdot59\cdot307)=88$,&&\\
&$h(-3\cdot7\cdot103)=8,\ h(-3\cdot5^2\cdot7\cdot103)=48$,&&\\
&$h(-4\cdot3\cdot53\cdot83)=48,\ h(-4\cdot5\cdot41)=8$,&&\\
&$h(-2^4\cdot5\cdot41)=16,\ h(-2^6\cdot5\cdot41)=32$,&&\\
&$h(-2^8\cdot5\cdot41)=64,\ h(-5\cdot7\cdot31\cdot47)=48$,&&\\
&$h(-17\cdot331)=28,\ h(-3^2\cdot17\cdot331)=56$,&&\\
&$h(-4\cdot3^2\cdot5\cdot17)=16$,&&\\
&$h(-2^4\cdot5\cdot17)=8,\ h(-2^4\cdot3^2\cdot5\cdot17)=32$,&&\\
&$h(-2^6\cdot5\cdot17)=16,\ h(-2^6\cdot3^2\cdot5\cdot17)=64$,&&\\
&$h(-4\cdot19\cdot71)=56,\ h(-4\cdot3^2\cdot19\cdot71)=112$,&&\\
&$h(-3\cdot7\cdot13\cdot19)=16,\ h(-3^3\cdot7\cdot13\cdot19)=48$,&&\\
&$h(-131\cdot353)=24,\ h(-4\cdot61\cdot181)=72$,&&\\
&$h(-2^3\cdot3\cdot5\cdot7\cdot13)=32$,&&\\
&$h(-2^5\cdot3\cdot5\cdot7\cdot13)=64,\ h(-3\cdot37\cdot373)=48$,&&\\
&$h(-3\cdot5\cdot109)=8,\ h(-3\cdot5^3\cdot109)=40$,&&\\
&$h(-2^9\cdot3)=16,\ h(-2^3\cdot3\cdot5^2)=8$,&&\\
&$h(-2^5\cdot3\cdot5^2)=16,\ h(-2^7\cdot3\cdot5^2)=32$,&&\\
&$h(-2^9\cdot3\cdot5^2)=64,\ h(-4\cdot193)=4$,&&\\
&$h(-4\cdot7^2\cdot193)=32,\ h(-5\cdot79\cdot89)=32$,&&\\
&$h(-3\cdot29\cdot397)=32,\ h(-4\cdot3\cdot7\cdot13\cdot29)=64$,&&\\
&$h(-2^3\cdot3\cdot17\cdot19)=32,\ h(-2^5\cdot3\cdot17\cdot19)=64$,&&\\
&$h(-3\cdot67\cdot139)=40,\ h(-5\cdot13\cdot419)=40$,&&\\
&$h(-2^3\cdot7\cdot107)=16,\ h(-2^5\cdot7\cdot107)=32$,&&\\
&$h(-4\cdot3\cdot5\cdot43)=16,\ h(-4\cdot3^3\cdot5\cdot43)=48$,&&\\
&$h(-5\cdot439)=16,\ h(-3^2\cdot5\cdot439)=32$,&&\\
&$h(-3^2\cdot43)=4,\ h(-7^2\cdot43)=8$,&&\\
&$h(-3^2\cdot7^2\cdot43)=32,\ h(-4\cdot17)=4$,&&\\
&$h(-4\cdot3^2\cdot17)=8,\ h(-4\cdot5^2\cdot17)=24$,&&\\
&$h(-4\cdot3^2\cdot5^2\cdot17)=48,\ h(-2^3\cdot113)=8$,&&\\
&$h(-2^5\cdot113)=16,\ h(-2^7\cdot113)=32$,&&\\
&$h(-23\cdot461)=12,\ h(-3\cdot7\cdot463)=24$,&&\\
&$h(-2^3\cdot3\cdot59)=16,\ h(-2^5\cdot3\cdot59)=32$,&&\\
&$h(-4\cdot3\cdot5\cdot79)=16$&&\\\hline
\end{tabular}
\end{center}
\label{tab:3}
\end{table}

Part c) of this Proposition relates to Hasse's results in \cite{hzer}, in which he determines the ramification for a prime divisor $p$ of $f$ in the extension 
$\Omega_f/\Sigma$, assuming $p$ is inert or ramified in $K = \mathbb{Q}(\sqrt{-d})$.  Here we have shown that the ramification index of $p$ in $\Omega_{pf}/\Omega_f$ is $e = p-1$ for the special fields $K = \mathbb{Q}(\zeta), \zeta = i$ or $\rho$, if $\left(\frac{d_K}{p}\right) = 1$ and $p \nmid 6f$; while in \cite{hzer}, the ramification index of $p$ in $\Omega_{pf}/\Omega_f$ (over $K$) is given as $e = p+1$ or $p$, according as $p$ is inert or ramified in $K$ (with a finite number of exceptional primes depending on $d$).  Thus, in the case of these two fields the ramification index is $e = p-\left(\frac{-d}{p}\right)$, for $p \nmid 6f$. \medskip

We finish by considering $p = 7$ again, but with $n = 3$. \medskip

\noindent {\bf Example ($p = 7$, cont'd).}  The calculation of the discriminants in $\mathfrak{D}_{3,7}$ proceeds as follows.
\begin{align*}
4 \cdot 7^3 &= 1^2 + 3 \cdot 457, \  &&h(-3 \cdot 457) = 12,\\
&= 6^2 + 2^3 \cdot 167, \  &&h(-2^3 \cdot 167) = 12,\\
&= 8^2 + 2^2 \cdot 3 \cdot 109, \  &&h(-3 \cdot 109) = 12, \ h(-4 \cdot 3 \cdot 109) = 12,\\
&= 13^2 + 3 \cdot 401, \  &&h(-3 \cdot 401) = 6,\\
&= 15^2 + 31 \cdot 37, \  &&h(-31 \cdot 37) = 6,\\
&= 20^2 + 2^2 \cdot 3^5, \ &&h(-3^5) = 3, \ h(-4 \cdot 3) = 1,\\
& &&h(-4 \cdot 3^3) = 3, \ h(-2^2 \cdot 3^5) = 9,\\
&= 22^2 + 2^3 \cdot 3 \cdot 37, \ &&h(-2^3 \cdot 3 \cdot 37) = 12,\\
&= 27^2 + 643, \ &&h(-643) = 3,\\
&= 29^2 + 3^2\cdot 59, \ &&h(-59) = 3, \ h(-3^2 \cdot 59) = 6,\\
&= 34^2 + 2^3 \cdot 3^3, \ &&h(-2^3 \cdot 3) = 2, \ h(-2^3 \cdot 3^3) = 6,\\
&= 36^2 + 2^2 \cdot 19, \ &&h(-19) = 1, \ h(-2^2 \cdot 19) = 3.
\end{align*}
Hence, the left-hand sum in \eqref{eqn:1.5} for $p = 7, n = 3$ seems to be
$$\sum_{-d \in \frak{D}_{3,7}}{h(-d)} = 5 \cdot 12 + 9 + 4\cdot 6 + 5 \cdot 3 + 2 + 2 \cdot 1 = 112 = \frac{1}{3}(7^3-7).$$
To verify this, we need to check the discriminants arising from the middle equation in this list, namely, $20^2 + 2^2 \cdot 3^5$, and the conductors $f = 3^2, 2, 6$, and $18$, for the field $K = \mathbb{Q}(\rho)$.  To start with, we have from \eqref{eqn:8.3} that
$$\frak{p}_7 = (\rho^2(2-\rho)) = (-3-2\rho),$$
so that $\frak{p}_7$ splits in $\Omega_2 = K$, which we already knew.  Hence, the inertial degree of $\frak{p}_7$ in $\textsf{K}_T$ is either $1$ or $w/2 = 3$.
There are two $6$-th degree primitive irreducible polynomials dividing $R_3(x)$ (defined for $F(x,y)$), whose roots have period three, which are
\begin{align*}
f_1(x) &= x^6 - 5x^5 + 24x^4 - 33x^3 + 14x^2 - x + 1, \ \textrm{disc} = -2^4 \cdot 3^3 \cdot 5^6 \cdot 7^4\\
f_2(x) &= x^6 - 11x^5 + 78x^4 - 111x^3 + 38x^2 + 5x + 1, \ \textrm{disc} = -2^{12} \cdot 3^6 \cdot 7^4 \cdot 19^3.
\end{align*}
The second polynomial clearly corresponds to the discriminant $-d = -19$, so we check to see that $f_1(x)$ corresponds to $-d = -12$.  It is not hard to check that 
$2^4$ exactly divides the discriminant of the field generated by a root of $f_1(x)$, and the only order $\mathcal{O} \subset K$ with $h(\mathcal{O}) = 1$ and even discriminant is the order of discriminant $-d = -12$.  Thus, for $p = 7, f = 2$, the field $\textsf{K}_T = \mathbb{Q}(\eta) \subset L_{\mathcal{O},7} = \Sigma_7 \Omega_{14}$ is generated by a root of $f_1(x)$.  Since the period is $3$, this is also the inertial degree of $\frak{p}_7$ in $\textsf{K}_T/K$.  Thus, $-d = -4\cdot 3 \in \frak{D}_{3,7}$. \medskip

The conductors $f = 3^2, 6$ give class numbers equal to $3$, and should correspond to $18$-th degree factors of $R_3(x)$.  Three of the five $18$-th degree factors have discriminants divisible by odd powers of $19, 59$, and $643$, respectively, so the other two polynomials are the ones we consider:
\begin{align*}
f_3(x) &= x^{18} - 159x^{17} + 14667x^{16} - 262520x^{15} + 1827192x^{14} - 5511762x^{13}\\
& \ + 6400998x^{12} + 2368908x^{11} - 10788351x^{10} + 1092123x^9 + 16448067x^8 \\
& \ - 18151032x^7 + 7304576x^6 - 347994x^5 - 455628x^4 + 48656x^3 + 12117x^2\\
& \ + 141x + 1, \ \ \textrm{disc} = -2^{60} \cdot 3^{21} \cdot 5^{42} \cdot 7^{48} \cdot 11^{12} \cdot 17^{12} \cdot 59^{12} \cdot 101^6;\\
f_4(x) &= x^{18} - 513x^{17} + 1205109x^{16} - 25000092x^{15} + 194961366x^{14}\\
 & \  - 758185938x^{13} + 1806243666x^{12} - 2986781760x^{11} + 3766865571x^{10}\\
 & \  - 3813769291x^9 + 3004800795x^8 - 1648049760x^7 + 506411418x^6\\
 & \ - 19912194x^5 - 35772714x^4 + 5787300x^3 + 1196541x^2 + 495x + 1,\\
 & \  \textrm{disc} = -2^{192} \cdot 3^{37} \cdot 5^{48} \cdot 7^{48} \cdot 17^{18} \cdot 47^{12} \cdot 59^6\cdot 131^6.
\end{align*}
We check the criterion from Theorem \ref{thm:5}(e) for $f = 3^2, 6$.  We have
\begin{equation}
\frak{p}_7^3 = ((2-\rho)^3) = (-1+18\rho) = (18+19\rho) = (19 + \rho).
\label{eqn:8.4}
\end{equation}
Since
$$-1+18\rho \equiv -55 \equiv 1 \ (\textrm{mod} \ \frak{p}_7')$$
and
$$(1-18\rho)^2 \equiv 1 \ (\textrm{mod} \ (3)\frak{p}_7'),$$
we see that $\frak{p}_7^3 \in H'$, the ideal group corresponding to $\Sigma'$.  Since $\frak{p}_7^3$ also lies in the ideal group $H_f \subset K$ for $\Omega_f$, this shows that the discriminants $-3^5, -4 \cdot 3^3$ lie in $\mathfrak{D}_{3,7}$, and because \eqref{eqn:8.3} shows these discriminants do not correspond to periodic points of period $1$. Thus, the polynomials $f_3(x), f_4(x)$ must correspond to these two discriminants; in fact, $f_3(x)$ corresponds to $-d = -2^2 \cdot 3^3$, and $f_4(x)$ corresponds to $-d = -3^5$.  The same argument works for $f = 18$, clearly, so that the discriminant $-d = -2^2 \cdot 3^5$ corresponds to the unique irreducible factor of $R_3(x)$ of degree $54$. We see that luck is with us in this calculation, since $3$ divides all but one of the conductors corresponding to the field $K = \mathbb{Q}(\rho)$.  \medskip

Considering \eqref{eqn:8.4}, it only remains to rule out the conductor $f = 19$ and the discriminant $-3 \cdot 19^2$.  However, $h(-3 \cdot 19^2) = 6$, and should correspond to a factor of degree $36$ of $R_3(x)$.  However, Theorem \ref{thm:6} and the above calculations show that the four factors of $R_3(x)$ of degree $36$ correspond to the discriminants $-d = -3 \cdot 401, -31 \cdot 37, -3^2 \cdot 59$, and $-2^3 \cdot 3^3$, none of which have the form $-3f^2$.  Thus $-d = -3 \cdot 19^2 \notin \mathfrak{D}_{3,7}$.  \medskip

In fact, we see that the 18 discriminants in $\mathfrak{D}_{3,7}$ correspond exactly to the $18$ primitive irreducible factors of $R_3(x)$.  These arguments show that \eqref{eqn:1.5} is true for $p = 7$ and $n = 3$. \medskip

\noindent {\it The case $p = 7, n  = 4$.} \medskip

A similar analysis applies to the case $p = 7$ and $n = 4$.  As we see in Table 4, there are $49$ discriminants for which there is a solution of
$$4 \cdot 7^4 = x^2+dy^2, \ \ x \equiv \pm 1 \ (\textrm{mod} \ 7),$$
with the exponent $4$ being minimal.  There are two discriminants with $d_K = -3$ requiring special attention, $-d_1 = -3\cdot13^2$ and $-d_2 = -3^3\cdot13^2$ 
(starred in the table).  For $-d_2$, the criterion in Theorem \ref{thm:5}(e) shows that $-d_2 \in \mathfrak{D}_{4,7}$.  For $-d_1 = -507$ we argue as follows. The class number $h(-507) = 4$, and the four reduced quadratic forms $ax^2+bxy+cy^2$ of discriminant $-507$ are
$$(a,b,c) = (1,1,127), \ (7,5,19) \ (\textrm{order} \ 4), \ (3,3,43) \ (\textrm{order} \ 2), \ (7, -5, 19) \ (\textrm{order} \ 4).$$
Using \cite[Thm. 7.7, pp. 123-124]{co}, the equivalence classes of these forms correspond to proper ideals in $\mathcal{O} = \textsf{R}_{-507}$ of norms $1,3,7,21$ given by
\begin{align*}
\frak{a}_1 &= \left(1, \frac{9+13\sqrt{-3}}{2}\right), \ \ \frak{a}_7 = \left(7,\frac{9+13\sqrt{-3}}{2}\right),\\
\frak{a}_3 &= \left(3, \frac{9+13\sqrt{-3}}{2}\right), \ \ \frak{a}_{21} = \left(21, \frac{1167+13\sqrt{-3}}{2}\right).
\end{align*}
For the basis quotients $w_k$ for each of these ideals, we compute approximations to the values $\tau = w_k$ of the modular function
$$z(\tau) = \left(\frac{\eta(\tau/7)}{\eta(\tau)}\right)^4+8,$$
where $\eta(\tau)$ is the Dedekind $\eta$-function.  These are conjugate values in $\Omega_{13} = K(\sqrt{13+4\sqrt{13}})$ for the field $K = \mathbb{Q}(\sqrt{-3})$, and their common
 minimal polynomial is
\begin{align*}
m_{507}(X) &= X^8 + 30460X^7 + 597336466X^6 - 31824635456X^5 + 1183007853019X^4\\
& \ \  - 23762042987840X^3 + 243673134180850X^2 - 1224982889982500X\\
& \ \  + 2411525969910625.
\end{align*}
(See \cite[pp. 358, 361]{m7}.  We will show in another paper that the results in \cite[Section 5]{m7} are true in for discriminants $-d$ for which $\left(\frac{-d}{7}\right) = +1$, not just for the discriminants $-d = -l, -4l$ considered in \cite{m7}.)  This is a normal polynomial over $\mathbb{Q}$, with discriminant
$$\textrm{disc}(m_{507}(X)) = 2^{84}3^{44}5^{16}7^{24}13^6 17^8 29^4 41^6 47^4 59^2 83^2 311^4 479^2.$$
By Theorem \ref{thm:6}, the value $h(w_1/7)$ generates the inertia field $\textsf{K}_T$ of $\frak{p}_7$ in $L_{\mathcal{O},7}$, and its minimal polynomial is the polynomial
$$P_{507}(x) = x^8(x-1)^8 m_{507}\left(\frac{x^3-x+1}{x(x-1)}\right),$$
whose degree is $24$.  Moreover, the roots of $P_{507}(x)$ are periodic points for the algebraic function defined by $F(x,y) = 0$, and the period can be computed by reducing $P_{507}(x)$ modulo $7$:
\begin{align*}
P_{507}(x) &\equiv (x + 2)^{12}(x^4 + 6x^3 + 5x^2 + 4x + 3)(x^4 + 4x^3 + 3x + 4)\\
& \ \ \times (x^4 + 4x^3 + 3x^2 + 5x + 3) \ (\textrm{mod} \ 7).
\end{align*}
This shows that the inertial degree of the prime divisors of $\frak{p}_7$ in $\textsf{K}_T$ is indeed $4$ and $-d_1 = -507 \in \mathfrak{D}_{4,7}$, as claimed. In this case, the prime divisors of $\frak{p}_7$ in $F_{\frak{p}_7'} = \Omega_{13}$ split completely in $\textsf{K}_T$.  (Note that $\Sigma_{\frak{p}_7'} = K$ in this case.) \medskip

Now Theorem \ref{thm:6}  shows that each of the discriminants $-d$ whose class numbers are shown in Table \ref{tab:4} corresponds to a primitive irreducible factor of degree $6h(-d)$ of the polynomial $R_4(x)$ (for $F(x,y)$), and their total degree is $6 \cdot 784 = 2(7^4-7^2)$.  There are $49$ such discriminants, and therefore at least $49$ primitive irreducible factors of $R_4(x)$.  In fact, these are all the primitive irreducible factors of $R_4(x)$, which follows from the fact that
\begin{align*}
&R_n(x) \equiv -(x^{7^n}-x)(x+2)^{7^n-1} \ (\textrm{mod} \ 7), \ n \ge 1,\\
&\textrm{deg}(R_4(x)/R_2(x)) = 2\cdot 7^4 - 2\cdot 7^2 = 6\cdot 784.
\end{align*}
This shows that the discriminants in Table \ref{tab:4} are all the discriminants in $\frak{D}_{4,7}$ and verifies that Conjecture 1 holds for $p = 7$ and $n = 4$. \medskip

\begin{table}[htp]
\caption{Checking $p = 7$ for $n=4$.}
\begin{center}
\begin{tabular}{|c|c|c|c|c|}
\hline
$p$&Discriminants&LHS&RHS\\
\hline
$7$&$h(-11\cdot 97)=12,\ h(-3^2\cdot 11 \cdot97)=24$,&$784$&$784$\\
&$h(-2^3\cdot13 \cdot 23)=8,\ h(-2^5\cdot13 \cdot 23)=16$,&&\\
&$h(-2^2\cdot5\cdot53)=8,\ h(-2^2\cdot 3^2\cdot5\cdot53)=32$,&&\\

&$h(-3 \cdot 5\cdot17\cdot37)=32,\ h(-83\cdot113)=24$,&&\\

&$h(-2^2\cdot3\cdot13\cdot59)=48$,&&\\

&$h(-2^3\cdot3\cdot5\cdot19)=16,\ h(-2^5\cdot3\cdot5\cdot19)=32$,&&\\

&$h(-5\cdot71)=4,\ h(-5^3\cdot71)=20$,&&\\

&$h(-3\cdot23\cdot127)=24,\ h(-2^4\cdot3\cdot11)=8$,&&\\

&$h(-2^6\cdot3\cdot11)=16,\ h(-2^8\cdot3\cdot11)=32$,&&\\

&$h(-2^2\cdot31\cdot67)=16$,&&\\

&$h(-3\cdot19\cdot139)=16,\ h(-3\cdot5\cdot11\cdot47)=16$,&&\\

&$h(-2^2\cdot73)=4,\ h(-2^2\cdot 5^2 \cdot73)=24$,&&\\

&$h(-3\cdot37)=8,\ h(-2^2\cdot3\cdot37)=8$,&&\\

&$h(-2^4\cdot3\cdot37)=16,\ h(-2^6\cdot3\cdot37)=32$,&&\\

&$h(-17 \cdot43)=12,\ h(-3^2\cdot17 \cdot43)=24$,&&\\

&$h(-5\cdot31\cdot41)=16$,&&\\

&$h(-2^7\cdot5)=8,\ h(-2^3\cdot3^2 \cdot 5)=8$,&&\\

&$h(-2^5\cdot3^2 \cdot 5)=16,\ h(-2^7\cdot3^2 \cdot 5)=32$,&&\\

&$h(-2^2\cdot 17)=4,\ h(-2^2\cdot3^2\cdot 17)=8$,&&\\

&$h(-2^2\cdot3^4\cdot 17)=24,\ h(-29\cdot167)=8$,&&\\

&{\bf **}$h(-3\cdot13^2)=4,\ h(-3^3\cdot13^2)=12$,&&\\

&$h(-2^2\cdot3\cdot11\cdot29)=16$,&&\\

&$h(-5\cdot11)=4,\ h(-2^2\cdot5\cdot11)=4$,&&\\

&$h(-2^4\cdot 5\cdot11)=8,\ h(-2^6\cdot5\cdot11)=16$,&&\\

&$h(-3\cdot5\cdot181)=8,\ h(-3\cdot13\cdot61)=16$,&&\\

&$h(-2^3\cdot47)=8,\ h(-2^5\cdot47)=16$,&&\\

&$h(-2^2\cdot3\cdot5\cdot19)=16$&&\\\hline
\end{tabular}
\end{center}
\label{tab:4}
\end{table}

\noindent Dept. of Mathematics, University of Missouri – Columbia \smallskip

\noindent 304 Math Sci Bldg, 810 Rollins St., Columbia, MO, 65211 \smallskip

\noindent email: akkarapakams@missouri.edu \bigskip

\noindent Dept. of Mathematical Sciences \smallskip

\noindent Indiana University at Indianapolis (IUI) \smallskip

\noindent LD 270, 402 N. Blackford St., Indianapolis, IN, 46202 \smallskip

\noindent e-mail: pmorton@iu.edu

\end{document}